\documentclass[12pt,a4paper]{article}
\usepackage{amsfonts}
\usepackage{amsmath}
\usepackage{amssymb}
\usepackage{fullpage}
\usepackage{amsthm}
\usepackage{graphicx}

\newtheorem{theorem}{Theorem}[section]
\newtheorem{proposition}[theorem]{Proposition}
\newtheorem{lemma}[theorem]{Lemma} 
\newtheorem{corollary}[theorem]{Corollary}
\theoremstyle{remark}
\newtheorem{remark}[theorem]{Remark}

\newtheorem{example}[theorem]{Example}
\theoremstyle{definition} 
\newtheorem{definition}[theorem]{Definition}

\newcommand{\zz}{{\mathbb Z}}

\newcommand{\qq}{{\mathbb Q}}
\newcommand{\hh}{{\mathbb H}}
\newcommand{\bg}{{\mathbf{g}}}

\newcommand{\calr}{{\mathcal{R}}}

\newcommand{\diag}{{\mathrm{Diag}}}
\newcommand{\be}{{\bf e}}
\newcommand{\bv}{{\bf v}}
\newcommand{\ff}{{\mathbb F}}
\newcommand{\pp}{{\mathbb P}}

\title{Groups of type $FP$ via graphical small cancellation} 
\author{Thomas M. Brown 
\and
Ian J. Leary}

\date{\today}

\begin{document} 

\maketitle

\begin{abstract} 
We construct an uncountable family of groups of type $FP$.  In contrast
to every previous construction of non-finitely presented groups of type 
$FP$ we do not use Morse theory on cubical complexes; instead we use 
Gromov's graphical small cancellation theory.  
\end{abstract}

\section{Introduction} 

The first examples of non-finitely presented groups of type~$FP$ were 
constructed in the 1990's by Bestvina and Brady, using Morse theory 
on CAT(0) cubical complexes~\cite{bb}.  Brady also used similar 
techniques to construct finitely presented subgroups of hyperbolic 
groups that are not themselves hyperbolic because they are not 
$FP_3$~\cite{brady}.  With the benefit of hindsight, examples due 
to Stallings and Bieri of groups that are $FP_n$ but that are not 
$FP_{n+1}$ can be viewed as special cases of the Bestvina--Brady 
construction \cite{sta}~and~\cite[pp.~37--40]{bie}.  
In~\cite{buxgon}, Bux and Gonzalez 
pointed out the close connection between the Bestvina--Brady 
construction and Brown's criterion for finiteness properties~\cite{brown}.  
The computations of finiteness properties that Brown made in~\cite{brown}
using his new criterion can also be rephrased in terms of Morse theory.  
Since then, Morse theory on polyhedral complexes has been a vital
tool in computing the finiteness properties of many families of 
groups---for some notable recent examples see~\cite{buildings,etc}.  
The Bestvina--Brady argument has 
been extended in a number of ways; in particular the second named
author has constructed continuously many isomorphism types of groups
of type $FP$~\cite{fpg}.  Nevertheless, it is remarkable that until 
now, every construction of a non-finitely presented group of type 
$FP$ has relied on the same Morse-theoretic techniques that were 
used by Bestvina--Brady.  

There are of course a number of ways to make `new groups of type $FP$
from old': notable examples include the Davis trick, which has been
used to produce non-finitely presented Poincar\'e duality
groups~\cite{davpd,davbook,fpg}, the method used by
Skipper--Witzel--Zaremsky to construct simple groups with given
homological finiteness properties~\cite{swz} and the proof that every
countable group embeds in a group of type $FP_2$~\cite{sgfp2}.
Nevertheless, these constructions require a pre-existing family of
groups of type $FP$, and so they rely ultimately on the Morse
theoretic techniques of Bestvina--Brady.

Here, we give a new construction for non-finitely presented groups of
type $FP$, that relies on Gromov's graphical small cancellation
theory~\cite[Sec.~2]{gromov} instead of the techniques used by
Bestvina--Brady.  Gromov introduced graphical small cancellation as a
method to embed certain families of graphs, in particular an expanding
family, as subgraphs of a Cayley graph.  Our main idea is to use
graphical small cancellation to construct families of surjective group
homomorphisms with acyclic kernels.  

Our method naturally constructs uncountable families of groups, but
we claim that it is as simple as the method of~\cite{bb}, even if
one only wants to establish the existence of some non-finitely
presented group of type~$FP$.  Apart from graphical small cancellation, 
we use only classical tools from combinatorial and homological
group theory.  Some of the families of groups that we construct are
isomorphic to families from~\cite{fpg}, but some are rather different.  

The remainder of this introduction consists of a sketch of the
construction and the statements of our main results.  For the
definitions and background results that we assume concerning
homological finiteness properties and graphical small cancellation see
Section~\ref{sec:background}.

The Bestvina--Brady construction takes as its input a finite
flag simplicial complex $L$, which should be acyclic but not contractible
in order to produce a group $BB_L$ that is $FP$ but not finitely presented. 
For constructing continuously many generalized Bestvina--Brady groups 
$G_L(S)$ as in~\cite{fpg}, one also requires $L$ to be aspherical.  
Our construction takes as input a finite CW-complex $K$ that we
call a \emph{spectacular complex}, together with an infinite set 
$Z$ of non-zero integers.  In the definition, we refer to the
1-~and~2-cells of the complex as `edges' and `polygons' respectively.  

\begin{definition}\label{defn:kdefn}
A \emph{spectacular complex} $K$ is a finite 2-dimensional CW-complex $K$ with 
the following properties.

\begin{enumerate} 
\item The $1$-skeleton $K^1$ of $K$ is a simplicial graph;

\item The attaching map for each polygon $P$ of $K$ is an embedding of 
  the boundary circle $\partial P$ into $K^1$; 

\item Every edge path between distinct vertices of $K^1$ of valence
  at least three has length at least 5; 
  
\item The girth $g$ of the graph $K^1$ is at least $13$; 
  
\item The perimeter $l_P$ of each polygon $P$ satisfies $l_P>2g$;  
  
\item The polygons of $K$ satisfy a $C'(1/6)$ condition: for 
each pair $P\neq Q$ of polygons of $K$, each component of the 
intersection $\partial P\cap \partial Q$ of their boundaries 
contains strictly fewer than $\min\{l_P/6,l_Q/6\}$ edges; 

\item $K$ is acyclic.  
\end{enumerate} 
\end{definition} 

We emphasize that a spectacular complex is required to be 
2-dimensional: a finite tree (with no polygons) is not spectacular.  
The word spectacular is intended to be an acronym: Simplicial 1-skeleton,
Polygons Embed, $C'(1/6)$, Two-dimensional, ACyclic, with lower bounds 
on maximal Unbranching paths, Lengths or perimeters of polygons And
Rotundity, where `rotundity' is used as a replacement for `girth'.  
The existence of a spectacular complex $K$ will be established in
Section~\ref{sec:complex}.  
For now we suppose that we are given a 2-complex $K$ as above and an
infinite set $Z\subseteq \zz-\{0\}$.  We consider $K$ and $Z$ to be
fixed throughout and they will usually be omitted from the notation.  

For each subset $S\subseteq Z$, we define a group $H(S)=H(K,Z,S)$
whose generators are the directed edges of $K$, where the two
orientations of the same edge are mutually inverse group elements.
The relators between these generators depend on $Z$ and $S$ as well as
$K$, in the following way.  For each $n\in Z-S$ and each polygon $P$
of $K$, the word spelt around the `degree $n$ subdivision of $\partial
P$' is a relator.  For each $n \in S$ and each simple cycle $C$ in the
graph $K^1$, the word spelt around the `degree $n$ subdivision of $C$'
is a relator.

The group $H(S)$ is better understood via its graphical presentation.
Recall that a graphical presentation arises from a labelled graph.
The group presented by a graphical presentation has generators the
edge labels and relators the words that represent the labellings of
cycles in the graph.  Each component of the labelled graph is called
a graphical relator.  For $H(S)$, the set of labels used is the directed
edges of $K$, and the graphical relators are certain subdivisions of
$K^1$ and of the boundaries of polygons $P$ of $K$, with 
the canonical choice of labelling.
In more detail, for each $n\in Z-S$ and each polygon $P$ of $K$, the
`degree $n$ subdivision of $\partial P$' is a graphical relator and
for each $n\in S$ the `degree $n$ subdivision of $K^1$' is a
graphical relator.
The degree $n$ subdivision of a labelled graph is defined in
Section~\ref{sec:background}, but see also the example in
Figure~\ref{fig:one}.  

For $P$ a polygon of $K$, we define $H_P(S)$ to be the subgroup of
$H(S)$ generated by the directed edges of $P$.

\begin{theorem}\label{thm:main} 
For each $S\subseteq Z$ the graphical presentation for $H(S)$ given 
above satisfies the graphical small cancellation condition $C'(1/6)$. 
For each $S\subsetneq T\subseteq Z$ the natural bijection between 
generating sets extends to a surjective group homomorphism 
$H(S)\rightarrow H(T)$, whose kernel $K_{S,T}$ is a non-trivial 
acyclic group.  For each polygon $P$ the intersection 
$K_{S,T}\cap H_P(S)$ is the trivial group
and so the natural map restricts to an isomorphism $H_P(S)\cong H_P(T)$.  
\end{theorem} 

None of the groups $H(S)$ are of type $FP$, however they form the main
building block in our construction of such groups.  

\begin{figure}
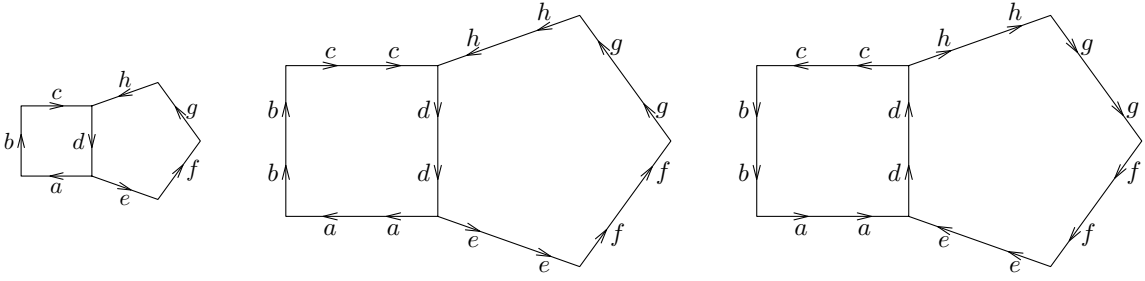

\begin{center}
  \raisebox{-0.5\height}{\includegraphics[height=18mm]{rel1.mps}}\qquad
  \raisebox{-0.5\height}{\includegraphics[height=36mm]{rel2.mps}}\qquad
  \raisebox{-0.5\height}{\includegraphics[height=36mm]{rel3.mps}}
\end{center}
\label{fig:one}
\caption{A graphical relator and its degree $2$ and $-2$ subdivisions.}
\end{figure}

In Proposition~\ref{prop:hpinjects}, as part of the proof of
Theorem~\ref{thm:main}, it will be shown that the isomorphism type of
the polygon subgroup $H_P$ depends only on $Z$ and on the perimeter of
the polygon $P$.  If $P$ has perimeter $l=l_P$ with $a_1,\ldots,a_l$
being the directed edge path around the boundary of $P$, 
then $H_P$ has the following presentation:
$$H_P=\langle a_1,\ldots, a_l\,\,\colon\,\, a_1^na_2^n\cdots a_l^n=1, \,\,\,
n\in Z\rangle.$$ 
The final ingredient that we need is an embedding of $H_P$ into a group 
$G_P$, where $G_P$ is of type $F$.  The existence of such an embedding
puts some constraints on the set $Z$.  We give explicit constructions for
$G_P$ in the cases $Z=\zz-\{0\}$ and $Z=\{k^n\,\colon\, n\geq 0\}$ for any 
integer $k$ with $|k|>1$.  Using Sapir's version of the 
Higman embedding theorem~\cite{sapir}, we show that $H_P$ embeds
in a group of type $F$ if and only if $Z$ is recursively enumerable.
However, we emphasize that our construction requires only one example
of such a $Z$, and so does not require Sapir's landmark result.  

Given such an embedding, we define a group $G(S)$ as the fundamental 
group of a star-shaped graph of groups.  The underlying graph has one 
central vertex and arms of length one indexed by the polygons of $K$. 
The central vertex group is $H(S)$, with $H_P=H_P(S)$ on the edge 
indexed by $P$ and $G_P$ on the outer vertex indexed by $P$.  

\begin{proposition} \label{prop:main} 
  Let the groups $G(S)$ for $S\subseteq Z\subseteq \zz$ be defined as
  described above, for any spectacular complex $K$, any recursively
  enumerable $Z\subseteq \zz$, and for each polygon $P$ of $K$ any
  fixed embedding of the subgroup $H_P$ into a group $G_P$ of
  type~$F$.

  In each such case the group $G(\emptyset)$ is of type
  $F$.  For each $S\subsetneq T\subseteq Z$ there is a surjective
  group homomorphism $G(S)\rightarrow G(T)$ with non-trivial acyclic
  kernel.
\end{proposition} 

\begin{corollary} \label{cor:main}
  Fix a spectacular complex $K$, a subset $Z$ of $\zz$, and embeddings
  of the groups $H_P$ into groups $G_P$ of type~$F$ as in the
  statement of Proposition~\ref{prop:main}.

  For each such set of
  choices there are continuously many isomorphism types of the groups
  $G(S)$ for varying $S\subseteq Z$.  Each group $G(S)$ is of type
  $FP$.  $G(S)$ is finitely presented if and only if $S$ is finite,
  and $G(S)$ embeds as a subgroup of a finitely presented group if and
  only if $S$ is recursively enumerable.  Provided that each $G_P$ has
  geometric dimension two, $G(\emptyset)$ also has geometric dimension
  two and each $G(S)$ has cohomological dimension two.
\end{corollary} 

Corollary~\ref{cor:main} should be compared with theorems 1.2~and~1.3
from~\cite{fpg}, in the special case when the flag complex $L$ used 
there is acyclic and aspherical.

Although the proofs are rather different, there is an overlap between
the groups $G(S)$ obtained from the new construction and the generalized
Bestvina--Brady groups of~\cite{fpg}.  In the case when $Z=\zz-\{0\}$, 
each $H_P$ is isomorphic to a Bestvina--Brady group and we may take 
for $G_P$ the corresponding right-angled Artin group.  For these 
choices, the group 
$G(S)$ is naturally isomorphic to the generalized Bestvina--Brady group 
$G_L(S\cup\{0\})$ of~\cite{fpg}, where $L$ is the flag triangulation 
of $K$ obtained by viewing each polygon as a cone on its boundary.  
In particular with these choices $G(\zz-\{0\})$ is isomorphic to 
the Bestvina--Brady group $BB_L$.

In contrast, in the case when $Z=\{k^n\,\colon\, n\geq 0\}$ for $|k|>1$, 
our choice for $G_P$ leads to a group $G(S)$ in which each generator
for $H(S)$ is conjugate to its own $k$th power.  Hence any semisimple
action of $G(S)$ on a CAT(0) space will have $H(S)$ in its kernel,
indicating that these groups are very different to generalized
Bestvina--Brady groups.  

In the next section, we give some background material concerning
finiteness properties, graphical presentations and graphical small
cancellation.  Most of this section is well-known material, but
we give some foundational results concerning graphical presentations
for which we have been unable to find a reference.  In  
Section~\ref{sec:smallc} we use graphical small cancellation methods
to prove Theorem~\ref{thm:main}.  In Section~\ref{sec:graphgp} we use
standard methods from graphs of groups to prove
Proposition~\ref{prop:main} and deduce Corollary~\ref{cor:main}.  In
Section~\ref{sec:embed} we discuss embeddings of the polygon groups
$H_P$ into 2-dimensional groups $G_P$ of type~$F$.  In 
Section~\ref{sec:complex} we establish the existence of a 2-complex
$K$ with the required properties, using some background material
concerning projective linear groups that is described in 
Section~\ref{sec:pgltwo}.  The ordering of the material  
reflects the history of our work: in particular we had a rough version
of the main theorem long before we had established the existence of a
spectacular complex~$K$.  Finally, Section~\ref{sec:ques} discusses some
questions that remain open.

This work was done while the first named author was working on his PhD
under the supervision of the second named author.  Further properties
of the groups $G(S)$ and other methods for constructing spectacular complexes 
$K$ can be found in the PhD thesis of the first named
author~\cite{brownphd}.  The main motivation for this
work was the observation that the relations in the presentations for
the groups $G_L(S)$ given in~\cite[Defn.~1.1]{fpg} consist of a large
family of `long' relations together with a finite number of `short'
relations.  Another motivation (which predates the work in~\cite{fpg})
was a conversation between the second named author and Martin Bridson,
in which Martin Bridson pointed out that there ought to be other
constructions of non-finitely presented groups of type~$FP$ apart from
that of Bestvina--Brady.  The authors gratefully acknowledge this
inspiration.  The authors also gratefully acknowledge helpful comments
on this work by Tim Riley.  Finally, the authors thank the anonymous
referee, whose numerous comments on four earlier versions have greatly
improved the exposition and led to the correction of some minor errors.  

\section{Background} 
\label{sec:background}

We begin with some remarks concerning 2-dimensional CW-complexes
whose 1-skeleta are simplicial graphs and such that the attaching
maps for 2-cells are embeddings, i.e., complexes that satisfy
conditions 1~and~2 from the definition of a spectacular complex.
Two types of subdivision of such complexes will be of interest.
Firstly, for $m>1$, there is a subdivision in which each
1-cell $e$ is subdivided into $m$ subintervals, with $m-1$
new 0-cells at their intersections.  In this subdivision the attaching 
maps for the 2-cells are unchanged, although the length of each 2-cell
is multiplied by $m$.  For $m\geq 5$ conditions 3~and~4 in the definition
of a spectacular complex will always hold for this subdivision.  
Secondly, if we replace each 2-cell by the cone
on its boundary, we obtain a simplicial complex that is homeomorphic
to the original complex, with one new vertex for each old 2-cell.
We call this simplicial complex the \emph{conical subdivision} of the
original complex.  

Although we do not need this result, we digress to explain the
connection between the fundamental group of a spectacular 2-complex
and a standard small cancellation group.

\begin{proposition}
  Let $K$ be a $2$-complex satisfying conditions $1$, $2$ and $6$ from
  the definition of a spectacular complex and having $n$ vertices.  Let
  $G$ be the group with inverse pairs of generators given by the edges
  of $K$ and relators the words along the polygons of $K$.  Then $G$
  is a classical $C'(1/6)$ small cancellation group and there is
  an isomorphism $\pi_1(K)*F\cong G$, where $F$ is a
  free group of rank $n-1$.
\end{proposition}

\begin{proof}
  Let $T$ be the cone on the 0-skeleton $K^0$ of $K$, so that $T$ is
  a star with $n$ arms,
  and let $X$ be the 2-complex defined as the union of $K$ and $T$,
  identifying $K^0$.  It is easy to see that $X$ is homotopy
  equivalent to the 1-point union of $K$ and the complete bipartite
  graph $K(2,n)$, and so $\pi_1(X)\cong \pi_1(K)*F$.  To complete the
  proof it suffices to show that $\pi_1(X)$ is isomorphic to the group
  $G$ described in the statement.  Since $T$ contains every vertex of $X$,
  it is a maximal tree in $X$.  Using this maximal tree as our starting
  point we obtain a presentation of $\pi_1(X)$ with generators the
  directed edges of $X$ not contained in $T$ and with relators corresponding
  to the 2-cells of $X$.  This presentation is exactly the one given for the
  group $G$, and this presentation is $C'(1/6)$.  
\end{proof}

An \emph{Eilenberg--Mac~Lane space} for a group $G$ is a connected CW-complex
whose fundamental group is isomorphic to $G$ and whose universal cover
is contractible.  Any two such spaces are based homotopy equivalent.

A group $G$ is \emph{of type $F$} if $G$ admits an Eilenberg--Mac~Lane
space with finitely many cells.  A space is \emph{acyclic} if it has
the same homology as a point.  A group $G$ is \emph{of type $FH$} if
there is a free $G$-CW-complex that is acyclic and has only finitely
many orbits of cells.  A group $G$ is \emph{of type $FL$} if the trivial
module $\zz$ for its group algebra $\zz G$ admits a finite resolution
by finitely generated free $\zz G$-modules.  Finally, a group $G$ is
\emph{of type $FP$} if $\zz$ admits a finite resolution by finitely
generated projective $\zz G$-modules.  From the definitions it is easy
to see that
$$F\implies FH \implies FL \implies FP,$$
and it may be shown that any finitely presented
group of type $FL$ is of type $F$.
For further details concerning these properties 
see~\cite{bie,brobook} and~\cite[Sec.~1]{bb}.  
We require one general result concerning finiteness properties.
A group is said to be \emph{acyclic} if its Eilenberg--Mac~Lane space is 
acyclic.  

\begin{proposition}
  \label{prop:coh} 
  Suppose that $G$ is of type $F$ and that $N$ is an acyclic normal
  subgroup of~$G$.  The group $G/N$ is of type $FH$, and the
  cohomological dimension of $G/N$ is bounded above by the
  geometric dimension of $G$.
\end{proposition}

\begin{proof}
  Let $X$ be the universal covering space of an Eilenberg--Mac~Lane
  space for $G$, and consider the quotient $X/N$.  This is an
  Eilenberg--Mac~Lane space for $N$, equipped with a free cellular
  action of $G/N$.  Since $N$ is acyclic, $X/N$ is acyclic.  The
  $G/N$-orbits of cells in $X/N$ correspond to the $G$-orbits of
  cells in $X$, and so if $G$ is of type $F$ then $G/N$ is of type $FH$.
  If $X$ has dimension $n$ then so does $X/N$, and the dimension
  of $X/N$ is an upper bound for the cohomological dimension of
  $G/N$.
\end{proof}

A \emph{graph of groups} indexed by a graph $\Gamma$ consists of groups
$G_v$ and $G_e$ for each vertex $v$ and edge $e$ of $\Gamma$,
together with two injective group homomorphisms $G_e\rightarrow G_v$
from the edge group $G_e$ to the vertex groups corresponding to the
ends of $e$.  A \emph{graph of based spaces} is defined similarly.  If each
vertex and edge space in a graph of spaces is an Eilenberg--Mac~Lane
space and the induced maps on fundamental groups are all injective,
then the homotopy colimit (in the category of unbased spaces) of
the graph of spaces is also an Eilenberg--Mac~Lane
space, whose fundamental group is by definition the \emph{fundamental
  group of the graph of groups}.  See~\cite[Ch.~1.B]{hatcher} for a
treatment of this topic, and either
\cite[Appendix]{ijlrs}~or~\cite[Sec.~4]{dror} for a treatment that
mentions homotopy colimits.  We highlight two special cases of the
Eilenberg--Mac~Lane space for the fundamental group of a graph of
groups that will appear in our work.  See also Figure~\ref{fig:two}.   

\begin{proposition}
  \label{prop:graphofgroups}
  \begin{enumerate}
  \item{}
  Suppose that $X$ is an Eilenberg--Mac~Lane space for the group $G$
  and that $f:X\rightarrow X$ is a based map that induces an injective
  homomorphism $\phi:G\rightarrow G$.  Then the mapping torus of $f$
  is an Eilenberg--Mac~Lane space for the ascending HNN-extension
  $\langle G,t \,: \, tgt^{-1}=\phi(g),\,g\in G\rangle$.  
\item{}
  Fix an indexing set $I$ and suppose that $X$, $Y_i$ and $Z_i$ are
  Eilenberg--Mac~Lane spaces for groups $H$, $H_i$ and $G_i$ respectively.
  Suppose also that $f_i:Y_i\rightarrow X$ and $g_i:Y_i\rightarrow Z_i$
  are based maps that induce injective group homomorphisms
  $\phi_i:H_i\rightarrow H$ and $\psi_i:H_i\rightarrow G_i$.
  In this case the star-shaped graph of spaces given as the
  following identification space 
  \[ \frac{\left( X \sqcup \bigsqcup_{i\in I} (Y_i\times [0,1]) \sqcup
  \bigsqcup_{i\in I} Z_i\right)}{(y_i,0)\sim f_i(y_i),\,\, (y_i,1)\sim g_i(y_i)}\]
  is an Eilenberg--Mac~Lane space for the corresponding star-shaped
  graph of groups.
  \end{enumerate}
\end{proposition}

\begin{figure}
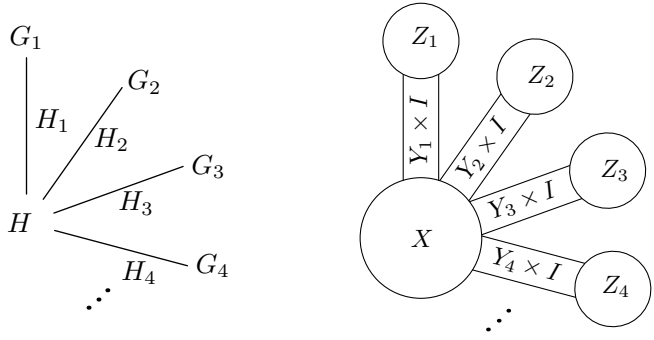

\begin{center}
  \raisebox{-0.5\height}{\includegraphics[height=38mm]{graph1.mps}}\qquad
  \qquad 
  \raisebox{-0.5\height}{\includegraphics[height=44mm]{graph2.mps}}
\end{center}
\label{fig:two}
\caption{A star-shaped graph of groups and its Eilenberg--Mac~Lane space.}
\end{figure}

Next we describe \emph{graphical small cancellation theory} as
in~\cite[Sec.~2]{gromov}, \cite{ollivier}~and~\cite{gruber}.
The theory subsumes 
the classical small cancellation theory~\cite[Ch.~V]{lynsch},
which is the case in which the graph $\Gamma$ considered below is a
disjoint union of cycles.  Note that we consider only the version
that yields torsion-free groups, which corresponds in the classical
case to excluding the possibility that a relator is a proper power.
This is the only case discussed by Ollivier~\cite{ollivier}.
Gruber discusses the more general case, but uses terms like 
`$C'(1/6)$' and `$C(7)$' in the case that yields only torsion-free
groups and terms like `$Gr'(1/6)$', `$Gr(7)$' for the more general
case~\cite{gruber}.  

Before introducing the graphical small cancellation conditions
we start by discussing graphical presentations of groups and
the associated graphical presentation complexes.  

A \emph{labelling} of a graph $\Gamma$ consists of a set $L$ of
labels, together with a fixed-point free involution $\tau:L\rightarrow
L$, and a labelling function $\phi$ from the set of directed edges
of $\Gamma$ to $L$ so that the label on the opposite edge to $e$ is
$\tau\circ\phi(e)$.  We emphasize that when $\Gamma$ is viewed as a
topological space, each pair of opposite directed edges corresponds
to a single 1-cell of $\Gamma$ with its two orientations, rather than
two distinct 1-cells.  
A labelling is said to be \emph{reduced} if the graph
$\Gamma$ contains no vertices of valence 0~or~1 and, for every vertex
$v$, the labelling function $\phi$ is injective on the outward-pointing
edges incident on $v$.  If a labelling is reduced then any word in the
elements of $L$ will describe at most one edge path starting at any
vertex $v$ of $\Gamma$.  A reduced word in $L$ is a finite sequence
of elements of $L$ that contains no subword $(l,\tau(l))$.  Conversely,
if $\phi:\Gamma\rightarrow L$ is a reduced labelling, then any directed
edge path in $\Gamma$ can be described by its initial vertex and
a word in the elements of $L$.  

A \emph{graphical presentation} is a 4-tuple $(\Gamma,L,\tau,\phi)$,
where $\Gamma$ is a graph and $\phi$ is a reduced labelling of
$\Gamma$ by $L$.  The group $G(\Gamma,L,\tau,\phi)$ 
presented by a graphical presentation
is the group given by the following ordinary presentation.  The
generators are the elements of $L$, subject to the following 
relations: for each $l\in L$, $\tau(l)$ is the inverse of $l$, 
and for every simple cycle in $\Gamma$, the word in $L$
obtained by going around that cycle is equal to the identity.
This definition is used in~\cite[Thm.~1]{ollivier}.

A reduced labelling of a graph $\Gamma$ by $L$ can be viewed as
defining an immersion from $\Gamma$ to the rose $R_L$, which is the
1-vertex graph with directed edges in bijective correspondence with
the elements of $L$.  We remind the reader that $R_L$ has $|L|/2$
distinct 1-cells.  The fundamental group of $R_L$ is of course
a free group with representatives of the $\tau$-orbits in $L$
as free generators.  The group 
$G(\Gamma,L,\tau,\phi)$ can be viewed as the quotient of the fundamental
group of $R_L$ by the normal subgroup generated by the images of the
fundamental groups of the components of $\Gamma$ under this immersion.
This is the definition of $G(\Gamma,L,\tau,\phi)$ used
in~\cite[Defn.~1.1]{gruber}.  It is immediate that the two
definitions describe the same group, essentially because the
fundamental group of a graph can always be generated as a normal
subgroup by simple cycles. 

The presentation 2-complex associated to a graphical presentation
in both \cite{ollivier}~and~\cite{gruber} relies on a choice of
free basis for the fundamental group of each component of $\Gamma$.
We prefer a different complex that requires no such choice, but we
shall show that our complex is homotopy equivalent to those used
in~\cite{ollivier,gruber}.  

A \emph{graphical $2$-cell} of the graphical presentation
$(\Gamma,L,\tau,\phi)$ is the cone on a component of $\Gamma$.
We emphasize that this definition does not appear
in~\cite{ollivier,gruber}.  
In the case when $\Gamma$ is a simplicial graph, each graphical
2-cell can be given the structure of a 2-dimensional simplicial
complex with one new vertex at the cone point, and with new edges
and triangles in bijective correspondence with the vertices and edges
(respectively) of the relevant component of $\Gamma$.  This is the
natural generalization of the conical subdivision of a
standard 2-cell.  The \emph{boundary} of a graphical 2-cell
is the base of the cone, a component of $\Gamma$.  

The \emph{graphical presentation $2$-complex} associated to a graphical
presentation $(\Gamma,L,\tau,\phi)$ is the 2-dimensional space
obtained by attaching the graphical 2-cells to the rose $R_L$,
using the immersions induced by $\phi$ to identify the boundary
of each graphical 2-cell with its image in $R_L$.  As was mentioned
in the introduction, a \emph{graphical relator} is just a component
of the labelled graph $\Gamma$.  Just as in the classical case,
the graphical presentation 2-complex is obtained by attaching
graphical 2-cells corresponding to the graphical relators to
a rose $R_L$.  

The \emph{graphical $1$-skeleton} of the graphical presentation
2-complex is just the rose $R_L$.  We emphasize that 
this will not usually be the whole 1-skeleton for any CW-structure
on the graphical presentation 2-complex.  

To show that our graphical presentation complex is naturally
homotopy equivalent to those used in~\cite{ollivier,gruber}
we will use the following proposition.  

\begin{proposition}\label{prop:allaresame} 
  Let $(\Gamma,L,\tau,\phi)$ be a labelled graph and for each
  component $\Gamma_i$ of $\Gamma$, suppose that $X_i$ and $Y_i$
  are contractible CW-complexes containing $\Gamma_i$ as a
  subcomplex.  Let $X$ (resp.~$Y$) be constructed by attaching
  each $X_i$ (resp.~$Y_i$) to $R_L$ using the map $\Gamma_i\rightarrow R_L$
  induced by $\phi$.  The complexes $X$ and $Y$ are homotopy
  equivalent relative to $R_L$.  
\end{proposition}

To prove this proposition we require one lemma.  Recall that a
\emph{CW-pair} $(X,A)$ is a pair consisting of a CW-complex $X$ and a
subcomplex $A$.

\begin{lemma}
  Suppose that $(X,A)$ and $(Y,A)$ are CW-pairs, and that $Y$ is
  contractible.  Then there is a map $f:X\rightarrow Y$ extending
  the identity map on $A$, and any two such maps are homotopic
  relative to $A$.

  If $X$ is also contractible, then $X$ and $Y$ are homotopy
  equivalent relative to $A$.
\end{lemma}

\begin{proof}
  As an inductive hypothesis, assume that $f$ is defined on $X^n\cup
  A$, the union of $A$ and the $n$-skeleton of $X$.  (The induction
  can be started with $n=-1$ and the convention that the $-1$-skeleton
  of $X$ is empty.)  It suffices to show that $f$ can be extended to
  each $(n+1)$-cell $\sigma$ of $X$ that is not contained in $A$.
  Since $Y$ is contractible, any map from an $n$-sphere to $Y$ extends
  to a map from the $(n+1)$-ball to $Y$.  Applying this statement to
  the map $f:\partial \sigma\rightarrow Y$ proves the claim.

  Given two maps $f,f':X\rightarrow Y$ that extend the identity
  map on $A$, construct a homotopy $H:X\times I\rightarrow Y$ by a similar
  inductive process.  Define $H(x,0)=f(x)$ and $H(x,1)=f'(x)$
  for all $x\in X$ and define $H(a,t)=a$ for all $t\in [0,1]$.
  Now suppose as an inductive hypothesis that $H$ is defined
  on $X\times\{0,1\}\cup A\times [0,1] \cup X^n\times [0,1]$.  
  It suffices to show that $H$ can be extended to $\sigma\times[0,1]$
  for each $(n+1)$-cell $\sigma$ of $X$.  But $H$ is already defined on
  the $(n+1)$-sphere
  $\partial(\sigma\times [0,1])= \sigma\times\{0,1\}\cup \partial
  \sigma\times [0,1]$, and since $Y$ is contractible this map can be extended to
  the $(n+2)$-ball $\sigma\times [0,1]$.

  For the second part, let $f:X\rightarrow Y$ and $g:Y\rightarrow X$
  be maps extending the identity on $A$.  Then $g\circ f:X\rightarrow
  X$ and the identity on $X$ both extend the identity on $A$, and so
  they are homotopic relative to $A$.  Similarly, $f\circ g$ and
  $1_Y$ are self-maps of $Y$ that extend the identity on $A$ so
  they are homotopic relative to~$A$.
\end{proof}

\begin{proof} (of Proposition~\ref{prop:allaresame}) By the lemma,
  for each component $\Gamma_i$ of $\Gamma$, there are maps
  $f_i:X_i\rightarrow Y_i$ and $g_i:Y_i\rightarrow X_i$ extending
  the identity on $\Gamma_i$.  There are also homotopies
  $g_i\circ f_i\sim 1_{X_i}$ and $f_i\circ g_i\sim 1_{Y_i}$ that
  fix $\Gamma_i$ throughout.  These maps and homotopies can be
  combined to give $f:X\rightarrow Y$ and $g:Y\rightarrow X$
  extending the identity map on $R_L$ and homotopies
  $g\circ f\sim 1_X$, $f\circ g\sim 1_Y$ that fix $R_L$ throughout.
  In more detail, the map $f$ is defined to be the identity on
  $R_L$ and to be $f_i$ on $X_i-\Gamma_i$.  The definitions of $g$
  and the homotopies are similar.  
\end{proof}

We are now ready to compare our graphical 2-complex with the ones
used in~\cite{ollivier,gruber}.  Both constructions, when applied
to a graphical presentation $(\Gamma,L,\tau,\phi)$ start with
the rose $R_L$.  For each component $\Gamma_i$ of $\Gamma$,
Ollivier and Gruber choose a family of (based) cycles $C_{i,j}$
that freely generate the fundamental group $\pi_1(\Gamma_i)$.
As subgraphs of $\Gamma$, these cycles are labelled via $\phi$,
and these maps are used as the attaching maps for 2-cells.  Unlike
our construction, this has the advantage that only ordinary 2-cells
are used, so they do not need to distinguish between the 1-skeleton
and the graphical 1-skeleton.  On the other hand, our construction
involves no choices.  

\begin{corollary}\label{cor:oursistheirs}
  For any graphical presentation $(\Gamma,L,\tau,\phi)$, our graphical
  presentation complex and the one used in~\cite{ollivier,gruber}
  are homotopy equivalent relative to $R_L$.
\end{corollary}

\begin{proof}
  Both constructions contain the rose $R_L$.  As described above,
  Gruber and Ollivier attach 2-cells to $R_L$ indexed by cycles
  $C_{i,j}$, where $\Gamma_i$ is a component of $\Gamma$ and
  for fixed $i$ the cycles $C_{i,j}$ form a free basis for
  $\pi_1(\Gamma_i)$.  The same space may be obtained by a
  2-stage process: first form a space $Y_i$ by attaching
  2-cells to $\Gamma_i$ along the cycles $C_{i,j}$, and then
  use $\phi$ to identify $\Gamma_i\subseteq Y_i$ with its image
  in $R_L$.  The fact that the cycles $C_{i,j}$ for fixed $i$
  form a free basis for $\pi_1(\Gamma_i)$ implies that $Y_i$
  is contractible.  Now define $X_i$ to be the cone on $\Gamma_i$.
  The hypotheses of Proposition~\ref{prop:allaresame} are satisfied
  with these choices of $X_i$ and $Y_i$, and the claimed result
  follows, since the $X$ in Proposition~\ref{prop:allaresame}
  is our graphical presentation 2-complex and the $Y$ is the
  one from~\cite{ollivier,gruber}.  
\end{proof}

The \emph{graphical Cayley complex} associated to a graphical
presentation $(\Gamma,L,\tau,\phi)$ is the universal covering
space of its graphical presentation complex.  The group
$G=G(\Gamma,L,\tau,\phi)$ 
presented by the graphical presentation acts freely on the
graphical Cayley complex.  The inverse image of the rose
$R_L$ in the universal covering space is by definition the
\emph{graphical $1$-skeleton} of the graphical Cayley complex.
The graphical 1-skeleton contains one free orbit of vertices
and its edges are labelled by $L$ in such a way that the action
of $G$ preserves the labels.  Thus the graphical 1-skeleton is
just the ordinary Cayley graph of $G$ with generators the
(inverse pairs of) elements of $L$.  
The graphical Cayley complex may be constructed by attaching
free $G$-orbits of graphical 2-cells to the graphical 1-skeleton.
In more detail, one free $G$-orbit of copies of the cone on
$\Gamma_i$ is attached for each component $\Gamma_i$ of $\Gamma$.

Having discussed graphical presentations and their graphical
presentation complexes, we move on to discuss the graphical
small cancellation condition.  

A \emph{piece} is a reduced word in $L$ 
that defines edge paths that start at at least two distinct
vertices of $\Gamma$.  
A piece is said to \emph{belong to} each component of $\Gamma$
that contains at least one of these vertices.  
The \emph{length} of a piece is the number of letters in the
word, or equivalently the length of the corresponding edge paths.

The \emph{girth} of a graph $\Gamma_i$ is the minimal length of any cycle in
$\Gamma_i$.  A graphical presentation satisfies the \emph{small cancellation
condition $C'(1/6)$} if the length of any piece is strictly less than
$1/6$ of the girth of each component to which the piece belongs.  We
are now ready to state the main theorem of graphical small
cancellation theory.  

\begin{theorem}\label{thm:mainofsc} 
  If the graphical presentation $(\Gamma,L,\tau,\phi)$
  satisfies $C'(1/6)$ then the associated graphical Cayley complex
  $X$ is contractible, and the attaching map for each graphical
  $2$-cell is an injection from a component of $\Gamma$ to the
  graphical $1$-skeleton of $X$.  Moreover, the Cayley graph
  $X^1$ has the `Dehn property': any closed loop in $X^1$ contains
  strictly more than half of some cycle of $\Gamma$ as a subpath.  
\end{theorem}

\begin{proof}
  By Corollary~\ref{cor:oursistheirs}, this statement will be true for
  our graphical presentation 2-complex if and only if it is true for
  the complex used in~\cite{ollivier,gruber}.  We shall describe how
  to deduce the above result from theorems in these two sources.  A
  version of this theorem is more readily seen in~\cite{ollivier}, but
  the version of parts of the theorem that can be found
  in~\cite{gruber} works with weaker hypotheses.  Note also that both
  \cite{ollivier}~and~\cite{gruber} make the statement that the
  graphical presentation complex is aspherical, rather than the
  equivalent statement that the graphical Cayley complex is
  contractible.

  Theorem~1 of Ollivier's article~\cite{ollivier} implies the above
  result, except for two points.  Firstly, because Ollivier is
  interested in hyperbolicity, Theorem~1 of~\cite{ollivier} includes
  the condition that $\Gamma$ should be finite.  This condition is
  removed in Remark~21 of~\cite{ollivier}, at the expense of also
  removing the conclusion that $G(\Gamma,L,\tau,\phi)$ is hyperbolic.
  Secondly, Theorem~1 of~\cite{ollivier} includes the extra hypothesis
  that $\Gamma$ should be `non-filamentous', i.e., every edge of
  $\Gamma$ should be contained in a non-trivial simple cycle.  This
  hypothesis is an artefact of the proof and is not needed; in
  remark~3~of~\cite{ollivier} it is asserted that this hypothesis is
  used only to ensure that components of $\Gamma$ inject into the
  graphical Cayley complex.  This part of the theorem is proved
  in~\cite{gruber} without this hypothesis.  

  Theorem~2.18~of~Gruber's article~\cite{gruber} proves that the
  graphical Cayley complex is contractible with our $C'(1/6)$
  condition replaced by the much weaker hypothesis $C(6)$ (no cycle is
  a union of six or fewer pieces).  Lemma~4.1 of~\cite{gruber} shows
  that the components of $\Gamma$ inject into the graphical 1-skeleton
  (which is just the Cayley graph of $G$ with respect to the
  generating set $L$).  Again, Lemma~4.1 of~\cite{gruber} uses the
  small cancellation property $Gr(6)$, which is even weaker than
  $C(6)$: see the discussion of these various conditions in and below 
  Definitions 1.2~and~1.3 of~\cite{gruber}.  Unfortunately for our
  purposes~\cite{gruber} does not discuss the Dehn property at all,
  but as stated above, this part of the theorem is proved
  in~\cite{ollivier}.
\end{proof}

\begin{remark} 
The reader who wishes to use only results concerning graphical
small cancellation that are stated in~\cite{ollivier} should note
that all of the graphs that we will use in the graphical presentation
for $H(S)$ will be non-filamentous in the sense of Ollivier provided
that the 1-skeleton $K^1$ of $K$ has no cut points.  This extra
hypothesis on $K$ does hold for the examples that we construct below.
\end{remark}

\begin{remark}
  We use the phrase `Dehn property' rather than `Dehn algorithm'
  because we are interested in the case when $\Gamma$ has infinitely
  many components, although each component will be a finite graph.
  In this case the Dehn property by itself does not suffice to produce
  an algorithm to decide whether a path in the Cayley graph is closed,
  i.e., to solve the word problem in the group.  The other ingredient
  needed is an algorithm to list all of the cycles in $\Gamma$ of
  at most a given length.  (This issue arises already in the classical
  case.)  
\end{remark} 

Before introducing the graphical presentation for the groups $H(S)$
that is central to our discussion, we discuss a well-known example
in which graphical small cancellation can be applied while classical
small cancellation does not apply directly.  

\begin{example}
  Let $T$ be a torus with a small open disc removed.  The fundamental
  group of $T$ is free of rank two, and with respect to a natural choice
  of generating set the bounding curve of the disc represents the
  commutator $[a,b]$ of free generators.  Now consider the space obtained
  by taking two copies of $T$ and identifying their boundary curves.
  The van Kampen theorem gives a presentation for the fundamental group
  of this space, which is of course a closed orientable surface of genus two: 
  \[\langle a,b,c,d\,\,:\,\,[a,b]=[c,d]\rangle.\]
  In the relator $[a,b][c,d]^{-1}$ the only pieces are single letters,
  and so this group presentation satisfies the $C'(1/6)$ condition
  (and even the $C'(1/7)$ condition).  The presentation 2-complex
  obtained by attaching the octagonal 2-cell to the rose with petals
  $a,b,c,d$ is homeomorphic to the given closed surface.  This surface
  admits a Riemannian metric of constant curvature $-1$, and for any
  such metric its universal covering space (which is the Cayley 2-complex
  for the given presentation) is identified with the
  hyperbolic plane $\hh$.  By choosing the metric as symmetrically as
  possible, one identifies the Cayley complex 
  with a tesselation of $\hh$ by regular hyperbolic
  octagons with interior angles $\pi/4$.  (Note that eight of these octagons
  will meet at each vertex of the tesselation.)  With respect to this metric,
  the circle where the two copies of $T$ were joined together represents
  a long diagonal of the octagon.  In the conical subdivision of the
  Cayley 2-complex, each regular octagon is replaced by eight
  isosceles triangles with angles $\pi/4,\pi/8,\pi/8$.  Each of the
  two copies of $T$ is made by identifying some of the sides of
  four of these triangles.  

  Now repeat the above, but instead of identifying the bounding circles
  of two copies of $T$, identify the bounding circles of \emph{three} copies.
  The fundamental group $G$ of the resulting space $X$ has the presentation
  \[G=\langle a,b,c,d,e,f\,\,:\,\,[a,b]=[c,d]=[e,f]\rangle.\]
  The classical small cancellation conditions fail dismally: there are
  three natural choices of octagonal relators, but any two of them
  intersect in a piece of length~4.  On the other hand, the symmetrical
  Riemannian metrics defined on each union of two copies of $T$ coincide
  on their intersections, yielding a locally CAT($-1$)-metric on $X$,
  so the fundamental group is hyperbolic.

  The solution is to take one graphical relator.  Take a graph $\Gamma$ 
  consisting of two vertices of valence three, joined by paths
  consisting of four edges; equivalently this is the barycentric
  subdivision of the complete bipartite graph $K(3,2)$.  If $x,y$
  are the two vertices of valence three, label the three paths
  from $x$ to $y$ by the three commutators $[a,b]$, $[c,d]$ and
  $[e,f]$.  In this single graphical relator, the pieces once
  again consist of single letters, and so the graphical $C'(1/6)$
  condition holds.  For this group, Ollivier's graphical presentation
  complex is obtained by attaching two of the possible octagons to
  the 6-petalled rose, whereas our graphical presentation complex
  is obtained by attaching the cone on $\Gamma$ to the 6-petalled
  rose.  In this case our graphical presentation 2-complex is
  homeomorphic to $X$, and even isometric to $X$ if we view the
  cone on $\Gamma$ as made from three halves of the regular octagon.
  The graphical Cayley complex is obtained from the Cayley graph of
  the group with respect to the generators $a,b,c,d,e,f$ by attaching
  a free $G$-orbit of cones on~$\Gamma$.

  Since $\Gamma$ is a simplicial graph, the conical subdivision of the
  graphical Cayley complex is a simplicial complex.  It has two free
  orbits of vertices.  Let $v$ be a lift of the vertex that covers the
  0-cell of the rose, and let $u$ be a lift of the cone vertex in the
  cone on $\Gamma$.  The link of the vertex $u$ is identified with
  $\Gamma$, so we may fix our choices in such a way that $v$ is
  adjacent to $u$ and furthermore the edge $\{u,v\}$ corresponds to
  the vertex $x\in \Gamma$.  There are~10 other vertices adjacent to
  $u$: these vertices are
  $av,abv,aba^{-1}v,[a,b]v,cv,cdv,cdc^{-1}v,ev,efv,efe^{-1}v$.  The
  group elements arising here are the elements represented by paths in
  $\Gamma$ from $x$ to another vertex.  The link of the vertex $v$ is
  larger: it contains twelve other vertices in the same orbit (the
  vertices $gv$ where $g$ is either a generator or the inverse of a
  generator) together with eleven vertices in the orbit of $u$.  Each
  2-simplex of the conical subdivision contains one vertex in the
  orbit of $u$ and two vertices in the orbit of $v$.  Realizing each
  2-simplex as a hyperbolic isosceles triangle with angles
  $\pi/4,\pi/8,\pi/8$ gives a $G$-equivariant CAT($-1$)-metric on the
  (conical subdivision of the) graphical Cayley complex, since
  Gromov's link condition~\cite[Ch.~II.5.24]{brihae} is easily
  verified.  The locally CAT($-1$)-metric that this induces on the
  quotient space is isometric to the locally CAT($-1$)-metric on $X$
  mentioned above.
  \end{example} 

Any graph $\Gamma$ admits a \emph{tautological labelling}, in which the
set $L$ is just the set of directed edges of $\Gamma$, the function
$\tau$ sends an edge to itself with the opposite orientation, and the
function $\phi$ is the identity map.  

Given a labelling of a graph $\Gamma$ and $n$ a non-zero
integer, the \emph{degree $n$ subdivision} of $\Gamma$ is the 
labelled graph obtained by subdividing each edge of $\Gamma$ into
$|n|$~parts.  If $n>0$, then the label attached to each of the $n$
new directed edges contained in the directed edge $e$ is $\phi(e)$,
whereas if $n<0$ the label attached to each new directed edge
contained in $e$ is $\tau\circ\phi(e)$.  It may be
helpful to imagine that the graph is rescaled by a factor of $|n|$,
so that the edges of the degree $n$ subdivision are `the same
length' as the edges of the original graph.  See Figure~\ref{fig:one}
for an example of a labelled graph and two of its subdivisions.  

We are now ready to define the graphical presentation for the groups
$H(S)$.

\begin{definition}\label{defn:hs}
Take a spectacular complex $K$, and take the tautological
labelling of its 1-skeleton $K^1$.
The generators for $H(S)=H(K,Z,S)$ are the
directed edges of $K$.  For each $n\in Z-S$ and for each polygon $P$
of $K$ the degree $n$ subdivision of the tautological labelling on
$\partial P$ is a graphical relator.  For each $n\in S$, the degree
$n$ subdivision of the tautological labelling on $K^1$ is a graphical
relator.  
\end{definition} 

\section{Using graphical small cancellation} 
\label{sec:smallc}

In this section we prove Theorem~\ref{thm:main}.  First we need
to establish that graphical small cancellation can be applied.  

\begin{proposition}\label{prop:hsissmallcanc} 
The given graphical presentation for $H(S)$ satisfies the 
small cancellation condition $C'(1/6)$.  
\end{proposition} 

\begin{proof} 
  This is a simple check.  
The shortest cycle in the degree $m$ subdivision of $K^1$ is of 
length $g|m|\geq 13|m|$, and the shortest cycle in the degree $m$ 
subdivision of the polygon boundary $\partial P$ is of length $l_P|m|> 
26|m|$.  Now suppose that $m\neq n$ are non-zero integers with 
$|m|<|n|$.  If $mn>0$, then 
the longest pieces contained in both a degree $m$ subdivision 
and a degree $n$ subdivision are of the form $a^mb^m$, where either $(a,b)$ 
or $(b,a)$ is a pair of consecutive edges in $K^1$.  If on the other 
hand  $mn<0$, then the longest pieces are of the form $a^m$.  

Between the degree $m$ subdivision of either $K^1$ or a polygon boundary
$\partial P$ and itself, the longest pieces are of the form $a^{|m|-1}$, 
of length $|m|-1$.  The longest pieces between the degree $m$ 
subdivisions of distinct polygon boundaries $\partial P$ and $\partial Q$ 
are potentially much longer, of length $|m|$ times the length of a
piece of $\partial P\cap \partial Q$.  But since the polygons of $K$
satisfy the $C'(1/6)$ condition, it follows that each piece of the
intersection of the degree $m$ subdivisions of $\partial P$ and
$\partial Q$ has length strictly less than both $|m|l_P/6$ and $|m|l_Q/6$.  
\end{proof} 

For $S\subseteq Z$, let $E(S)$ denote the standard graphical 
Cayley 2-complex for $H(S)$.  By Proposition~\ref{prop:hsissmallcanc}
and Theorem~\ref{thm:mainofsc} we see that $E(S)$ is contractible,
with a free action of $H(S)$, and hence that $E(S)/H(S)$ is an
Eilenberg--Mac~Lane space for $H(S)$.  Similarly, $E(T)/H(T)$ is
an Eilenberg--Mac~Lane space for $H(T)$.

For the proof of 
Theorem~\ref{thm:main} we will need to compare, for 
$S\subseteq T\subseteq Z$, two free $H(T)$-complexes: the 
standard complex $E(T)$ and the quotient $E(S)/K_{S,T}$.
(We remind the reader that $K_{S,T}$ is by definition the kernel
of the natural surjective homomorphism $H(S)\rightarrow H(T)$;
for any complex $E$ with an $H(S)$-action the quotient $E/K_{S,T}$
admits an $H(T)$-action.)

Since we know that $E(S)/H(S)$ is an Eilenberg--Mac~Lane space for
$H(S)$, the general theory of covering spaces tells us that 
$E(S)/K_{S,T}$ is an Eilenberg--Mac~Lane 
space for $K_{S,T}$ which admits a free cellular action of $H(T)$.
Define the graphical 1-skeleton of $E(S)/K_{S,T}$ to be the
image of the graphical 1-skeleton of $E(S)$, and define a
graphical 2-cell of $E(S)/K_{S,T}$ to be the image of a
graphical 2-cell of $E(S)$.  With this definition, the
graphical 1-skeleton of each of $E(T)$ and $E(S)/K_{S,T}$ is
the Cayley graph of $H(T)$ with respect to the same generating
set.  Thus these two free $H(T)$-spaces have the same graphical
1-skeleta.  They also share many of the same graphical 2-cells.  The
difference between them is that for each $n\in T-S$, 
$E(S)/K_{S,T}$ contains a free orbit of
polygonal 2-cells attached along the degree $n$ subdivision of $\partial
P$ for each polygon $P$ of $K$, whereas $E(T)$ contains a free orbit
of graphical 2-cells attached along the degree $n$ subdivision of the
graph $K^1$ itself.  It may be seen that the attaching map for each
graphical 2-cell is injective.  For the graphical 2-cells of $E(T)$
this follows directly from the small cancellation properties of the
presentation for $H(T)$.  For the 2-cells that belong only to $E(S)/K_{S,T}$,
i.e., the cells attached along the degree $n$ subdivision of $\partial P$
for $n\in T-S$, this follows from the fact that $\partial P$ maps injectively
into $K^1$, and the degree $n$ subdivision of $K^1$ is already known to
map injectively into the graphical 1-skeleton.  

There is a 3-dimensional $H(T)$-complex $F$ that
contains both $E(S)/K_{S,T}$ and $E(T)$ as subcomplexes.  The fundamental
theorem of graphical small cancellation tells us that inside $E(T)$ there
is a free $H(T)$-orbit of copies of cones on the degree $n$ subdivision
of $K^1$, for each $n\in T-S$.  These copies of the degree $n$ subdivision
of $K^1$ are of course also present inside $E(S)/K_{S,T}$ since the
graphical 1-skeleta are equal.  However, in $E(S)/K_{S,T}$,
instead of having a single cone vertex attached to each such copy,
there is a specific subcomplex $K_0$ homeomorphic to $K$, which has
this copy of the degree $n$ subdivision of $K^1$ as its 1-skeleton,
and in which each polygon $P$ of $K$ is attached to the degree $n$
subdivision of $\partial P$.  To make $F$ from $E(S)/K_{S,T}$ 
attach a free $E(T)$-orbit of cones on $K$ to the $E(T)$-orbit
of $K_0$.  
The inclusion of $K^1$ into $K$ induces an embedding
of the cone $C(K^1)$ on $K^1$ into the cone $C(K)$, and hence an
$H(T)$-equivariant inclusion $E(T)\rightarrow F$.  

The reader may prefer to understand the construction in terms of
the conical subdivisions of $E(T)$, $F$ and $E(S)/K_{S,T}$, which
are simplicial complexes with free $H(T)$-actions.  The graphical
1-skeleton of $E(T)$ and of $E(S)/K_{S,T}$ is the cover of the
rose $R_L$: it has one free orbit of vertices, with say $v$ as
an orbit representative.  The directed edges between vertices in this
orbit are labelled by elements of $L$, so that each vertex in
this orbit is adjacent to $|L|$ other vertices in this orbit.
For each $n\in S$, both $E(S)/K_{S,T}$ and $E(T)$ have another
orbit of vertices with orbit representative $u_n$.  The link
of each vertex in the orbit of $u_n$ contains only 
vertices in the orbit of $v$, and is a copy of the degree
$n$ subdivision of $K^1$.  For each $n\in Z-T$ and each polygon
$P$ of $K$, both $E(S)/K_{S,T}$ and $E(T)$ have an orbit of
vertices with orbit representative $x_{n,P}$.  The link of
each vertex in this orbit again contains only vertices in
the orbit of $v$, and is a copy of the degree $n$ subdivision
of $\partial P$.  The difference between $E(S)/K_{S,T}$ and
$E(T)$ arises for the vertices corresponding to $n\in T-S$.
For $n\in T-S$, $E(T)$ contains an orbit of vertices of type
$u_n$, with link a copy of the degree $n$ subdivision of $K^1$,
whereas $E(S)/K_{S,T}$ instead contains orbits of vertices of type
$x_{n,P}$ for each polygon $P$ of $K$,
whose links are copies of the degree $n$ subdivision of
$\partial P$.  In these terms, the conical subdivision of the
complex $F$ is obtained by
taking \emph{both} of these types of vertices: for $n\in T-S$,
the complex $F$ contains the vertices $x_{n,P}$ for each polygon
$P$, together with the vertices $u_n$.  The link of each $u_n$
vertex here contains exactly one vertex in the orbit of $x_{n,P}$,
as well as a number of $v$-vertices.  The link of each vertex in the
$u_n$ orbit for $n\in T-S$ is thus a 2-complex isomorphic to
the conical subdivision of the complex obtained from $K$ by
taking the degree $n$ subdivision of its 1-skeleton.

\begin{proposition} \label{prop:mainarg}
With notation as above, the inclusion $E(S)/K_{S,T}\rightarrow F$ is 
an $H(T)$-equivariant homology isomorphism and the inclusion
$E(T)\rightarrow F$ is an $H(T)$-equivariant homotopy equivalence,
whose image is an $H(T)$-equivariant deformation
retract of $F$.  
\end{proposition} 

\begin{proof}
The first claim follows because attaching a cone to an acyclic
subspace does not change homology.  For the second claim, note that
for a polygon $P$, the cone $C(\partial P)$ on $\partial P$ is a
deformation retraction of the cone $C(P)$.  Putting these deformation
retractions together, it follows that the cone $C(K^1)$ on
$K^1$ is a deformation retraction of the cone $C(K)$ on $K$.  Applying
this retraction simultaneously over all such cones appearing in $F$, 
it follows that $E(T)$ is an $H(T)$-equivariant deformation
retraction of $F$ as claimed.
\end{proof} 

\begin{proposition}\label{prop:hpinjects}
  For $P$ a polygon of $K$ with directed bounding cycle $a_1,\ldots,a_l$,
  let $H'_P$ be the group defined by the presentation
  \[H'_P=\langle a_1,\ldots,a_l\,\,:\,\, a_1^n\cdots a_l^n,\,\,n\in Z\rangle.\]
  The inclusion of generating sets induces an injective homomorphism
  $H'_P\rightarrow H(S)$ for each $S\subseteq Z$.  In particular, this
  map gives an isomorphism from $H'_P$ to $H_P(S)$, the subgroup of $H(S)$
  generated by the edges of $P$.  
\end{proposition}

\begin{proof}
  The relators in the given presentation for $H'_P$ all hold between
  the corresponding generators of $H(S)$, so there is a homomorphism
  $H'_P\rightarrow H(S)$ as claimed.  It remains to show that this
  homomorphism is injective.  Since $H(S)$ maps surjectively onto
  $H(Z)$, it clearly suffices to show the claim in the case $S=Z$.
  The given presentations for $H'_P$ and $H(S)$ satisfy the $C'(1/6)$
  property, so have the Dehn property.  To spell this out in greater
  detail, we recall that a word in the generators of a group (and
  their inverses) is said to be \emph{reduced} if it contains no
  subword of the form $aa^{-1}$ or $a^{-1}a$.  The Dehn property for
  the two given presentations is the following: any reduced word in
  the generators of $H'_P$ that is equal to the identity in $H'_P$
  contains strictly more than half of one of the defining relators as
  a subword, and any reduced word in the generators of $H(S)$ that is
  equal to the identity in $H(S)$ contains strictly more than half of
  the word spelt around a simple cycle in one of the defining
  graphical relators as a subword.

  Say that a word in the generators for $H'_P$ is \emph{$H'_P$-reduced}
  if it is reduced and does not contain more than half of any
  defining relator as a subword.  By induction on word length, the
  Dehn property shows that any word in the generators for $H'_P$
  represents the same element of $H'_P$ as some $H'_P$-reduced word. 
  Similarly, say that a word in
  the generators for $H(Z)$ is \emph{$H(Z)$-reduced} if it is
  reduced and does not contain more than half of any simple cycle
  in any of the labelled graphs used to define $H(Z)$ as a subword.
  Once again, applying the Dehn property and induction one sees that
  every word in the generators for $H(Z)$ represents the same group
  element as some $H(Z)$-reduced word.
  
  It suffices to show the image in $H(Z)$ of any non-trivial
  $H'_P$-reduced word is not the identity.  Not every $H'_P$-reduced
  word will be $H(Z)$-reduced, which complicates the argument
  considerably; however we claim that every
  non-trivial $H'_P$-reduced word is equal as an element of $H(Z)$
  to some non-trivial $H(Z)$-reduced word.  

  For the remainder of this proof, we temporarily redefine a \emph{piece} to
  be a reduced word that defines a path in at least \emph{one}
  of the (graphical) relators for the given presentation for $H(Z)$.
  Our justification for this term is that we will be considering
  pieces that arise as subwords of our given $H'_P$-reduced word, so
  the pieces that we will consider will still appear twice, once in a
  relator and once in our given word.

  Suppose that a pair $a,b$ of distinct directed edges of $K$ have
  exactly one vertex in common.  After possibly replacing each
  of the edges by its opposite, we may suppose that the terminal
  vertex of $a$ is equal to the initial vertex of $b$.  In this
  case the word $a^lb^m$ with $l,m\neq 0$ is a piece of the
  degree $n$ subdivision of $K^1$ if and only if $lm>0$,
  $mn>0$ and $|n| \geq |l|$, $|n| \geq |m|$.  This word is a
  piece of the degree $n$ subdivision in a unique way, since
  the division between the $a$-edges and the $b$-edges can
  only appear at one vertex.  

  Moving on to words in three letters, if $a,b,c$ is a directed
  edge path in $K^1$, and $l,m,n\in \zz-\{0\}$, then $a^lb^mc^n$
  can only appear as a piece of the degree $m$ subdivision of
  $K^1$, and will do so precisely when all of the following
  hold: $lm>0$; $mn>0$; $|l|\leq |m|$; $|n|\leq |m|$.
  Although we made the above statements for the subdivision of
  $K^1$, if the directed edges $a,b,c$ are contained in $P$ then
  the same statements hold for pieces of the subdivision of $P$.
    
  Given an $H'_P$-reduced word $w$, we consider how it breaks up into
  maximal pieces.  For example, if $w=ua^nv$ where the final letter of
  $u$ and the initial letter of $v$ have no vertex in common with $a$
  (when viewed as directed edges of $K$), 
  then $a^n$ is a maximal piece.  If a piece contains a subword
  $a^lb^mc^n$, where $a,b,c$ is a directed edge path in $P$ with
  $l,m,n\neq 0$, this piece can occur only in the degree $m$
  subdivision of $P$; in this case we call it a \emph{piece of degree} $m$.
  Short pieces, by which we mean pieces that consist of either a power
  of a single letter or a product of two powers of single letters, do
  not have a well-defined degree.  
  The maximal pieces of $w$ may intersect non-trivially.  If $m,n>0$
  with $m\neq n$ the intersection of a maximal piece of degree $m$
  and a maximal piece of degree $n$ may have length up to $2\min\{m,n\}$,
  with the worst case represented by $\ldots x^my^mz^ma^nb^nc^n\ldots$,
  where $x,y,z,a,b,c$ are consecutive directed edges (i.e., they form a
  directed edge path).
  If $m>0$ the intersection of two distinct maximal pieces of degree $m$
  can be at most length $m-1$, with the worst case represented by
  $\ldots x^my^mz^ma^{m+1}b^mc^m\ldots $ or by
  $\ldots x^my^mz^ma^{m-1}b^mc^m\ldots $.  These considerations show
  that it is possible for a maximal piece of well-defined degree to
  be entirely covered by its neighbours, but that the longest such
  maximal pieces consist of powers of at most four letters: if
  $m>0$ and $l,n>m$ then the word $\ldots x^ly^lz^ma^mb^nc^n\ldots$
  contains the maximal piece $y^mz^ma^mb^m$ which is entirely covered
  by the neighbouring pieces of degrees $l$~and~$n$.  In general, the
  length of 
  the intersection is not the important feature: what is crucial to
  our argument is that the intersection of maximal pieces of distinct
  degrees consists of at most two powers of letters, and that the
  intersection of two distinct maximal pieces of the same degree
  consists of a power of a single letter.  

  By definition, if $w$ is an $H'_P$-reduced word then $w$
  contains no degree $m$ pieces of length greater than $(m/2)l_P$.
  However, such a word will not necessarily be $H(Z)$-reduced: a
  piece of degree $m$ that is $H'_P$-reduced may be further
  shortened using a `shortcut' across the polygon $P$ that contains
  edges from $K^1$ that are not in $P$.  
  Such a shortcut will necessarily go between distinct
  vertices of $K^1$ of valence at least three.  If there is a maximal 
  degree $m$ piece in $w$ that is $H'_P$-reduced but not $H(Z)$-reduced,
  we replace it by its $H(Z)$-reduction.  Since this reduction involves
  a path between two vertices of $K^1$ of valence at least three that
  is not contained in $P$, condition~3 of Definition~\ref{defn:kdefn}
  implies that the $H(Z)$-reduction of this
  piece contains a subword of the form $r^ms^mt^mu^mv^m$, where $r,s,t,u,v$
  are consecutive directed edges of $K^1$ that are not contained in $P$.
  In particular, this piece can only appear in the degree $m$ subdivision
  of $K^1$, it cannot be entirely covered by its two neighbours, 
  and it has well-defined endpoints (vertices of the degree
  $m$ subdivision of $K^1$) that are equal to those of the piece that
  it replaced.  Provided that any two pieces that are long enough to
  support shortcuts are separated by smaller pieces, this shows that
  there can be no further cancellation once each piece with a
  shortcut has been reduced in this way.  

  The remaining difficulty is the case when two longer pieces both
  admitting shortcuts are adjacent, including the case when they
  overlap.  If these pieces have the same
  degree, $m$ say, then they are not compatible, in the sense that
  when they are both fitted into the degree $m$ subdivision of $K^1$,
  the intersection of their images there will not be precisely equal
  to their intersection in $w$; if this was not the case, then they
  would combine to form a single piece of degree $m$.  But now the
  same property also holds for their $H(Z)$-reductions, and so their
  reductions cannot be combined into a longer piece.  Note also that
  the intersection of two such pieces consists of at most a power of
  a single letter.  It remains to
  consider the case when two long pieces corresponding to different
  $m$ and $n$ are adjacent.  In this case, it is possible for their
  endpoints to match up and some cancellation may take place at
  their boundaries.  Viewing the terminal vertex of the first
  piece as a vertex of the degree $m$ subdivision of $P\subseteq K^1$
  and the initial vertex of the second piece as a vertex of the
  degree $n$ subdivision of $P\subseteq K^1$, the only potential
  problem is when these two vertices correspond to a single vertex
  of $K^1$ of valence at least three, as opposed to corresponding
  to vertices of the subdivision that appear somewhere in the middle
  of an edge of $K^1$.  Again, no reduction will occur unless 
  the same edge of $K^1-P$ appears with opposite sign at these
  ends.  For example, the first piece might end
  $r^ms^mt^mu^mv^m$ and the second piece might begin
  $v^{-n}u^{-n}t^{-n}s^{-n}r^{-n}$, with $m,n>0$.
  Even in this case, since $m\neq n$
  at most a power of one letter is lost, either from the end of the
  first piece (if $m<n$) or from the beginning of the second piece
  (if $m>n$), and so the first piece still has degree~$m$ and the
  second piece has degree~$n$.  Thus no two adjacent long pieces can
  combine to form a longer piece, and the reduction (in the free group)
  of the word that they spell together is $H(Z)$-reduced and non-trivial.
\end{proof} 

\begin{remark}
  The reader may prefer an alternative account of the above proof in
  terms of the simplicial complexes that are the conical subdivisions
  of the graphical Cayley complexes for $H'_P$ and for $H(Z)$.  The
  graphical 1-skeleton of this complex is the Cayley graph of the
  group $H'_P$ (resp.~$H(Z)$) with the directed edges of $P$
  (resp.~$K$) as generators.  For each $n\in Z$ there is an orbit of
  2-cells (resp.~graphical 2-cells) attached, with boundary the degree
  $n$ subdivision of $\partial P$ (resp.~the degree $n$ subdivision of $K^1$).
  In the conical subdivision this gives rise to an extra orbit of
  vertices for each $n\in Z$.  Let $v'$ (resp.~$v$) be an orbit
  representative of the vertices in the graphical 1-skeleton, and for
  $n\in Z$ let $u'_n$ (resp.~$u_n$) be an orbit representative of the
  vertices at the centre of the degree $n$ subdivision of $\partial P$
  (resp.~at the cone point in the cone on the degree $n$ subdivision
  of $K^1$).

  In this simplicial complex, a word is represented by an edge path
  that stays in the graphical 1-skeleton (i.e., that only passes
  through vertices in the orbit of $v'$ (resp.~$v$)).  A reduced word
  is such an edge path that never reverses its direction.  A piece (in
  the sense of the above proof) is a connected subpath that is
  contained in the link of one of the cone vertices (i.e., the
  vertices in the orbits $u'_n$ (resp.~$u_n$)).
  A piece consisting of
  just a power of a single letter is contained in the links of many
  different cone vertices.  A piece consisting of a product of two
  powers of letters is also contained in the links of many different
  cone vertices, but is contained in the link of at most one vertex in
  each orbit.  A piece of degree $n$ is contained in the link of a
  single vertex in the orbit of $u'_n$ (resp.~$u_n$) and in the link
  of no other vertex.  A reduced word is $H'_P$-reduced
  (resp.~$H(Z)$-reduced) if each of its pieces of degree $n$ consists
  of a shortest path in the link of its cone vertex between its two
  end points.

  The Dehn property tells us that any non-trivial edge path
  in the graphical 1-skeleton that is reduced and $H'_P$-reduced
  (resp.~reduced and $H(Z)$-reduced) has distinct end points.

  The homomorphism $H'_P\rightarrow H(Z)$ gives a local embedding
  of the simplicial complex for $H'_P$ into the simplicial complex
  for $H(Z)$, and the content of Proposition~\ref{prop:hpinjects}
  is that this map is actually an embedding.  To prove this we
  start with an edge path that is non-trivial, reduced and
  $H'_P$-reduced and consider its image in the (subdivided)
  graphical Cayley complex for $H(Z)$.  Since there may be
  shortcuts in $K^1$ between the ends of a shortest path in
  $P$, the image need not be $H(Z)$-reduced.  However, the only
  places where it can fail to be $H(Z)$-reduced are pieces that
  are sufficiently long that they have a well-defined degree.  
  For each such piece of degree~$n$, we start
  by replacing the piece by a shortest path in the degree $n$
  subdivision of $K^1$ between its end points.  If any two
  pieces of well-defined degree are separated by shorter pieces, then
  this replacement is already reduced and $H(Z)$-reduced.  If
  on the other hand there are pieces of well-defined degree that are
  adjacent in the path, it is possible that the $H(Z)$-reductions
  of these pieces intersect in such a way as to introduce
  back-tracking, so that the new path is no longer reduced.
  Removing the backtracking cannot reduce the number of powers
  of letters involved to fewer than three (since at least five letters
  were involved originally and at most one letter
  can be lost at each end of each piece), so removing backtracking
  cannot change a piece of degree~$n$ to a piece that is too short
  to have a well-defined degree.
  Hence the new path is non-trivial, reduced and $H(Z)$-reduced,
  which establishes that $H'_P$ embeds in $H(Z)$.
\end{remark}

\begin{proof} (of Theorem~\ref{thm:main})
Since the generating sets of $H(S)$ and $H(T)$ are identified with
each other, the homomorphism $H(S)\rightarrow H(T)$ is surjective.  To
see that its kernel is non-trivial, let $(a_1,\ldots, a_g)$ be a
directed loop in the graph $K^1$ of length equal to the girth of
$K^1$, let $n$ be an element of $T-S$, and consider the element
$h:=a_1^na_2^n\cdots a_g^n$ of $H(S)$.  This element is contained in
$K_{S,T}$ and we claim that it is not the identity element in $H(S)$.
By Proposition~\ref{prop:hsissmallcanc}, the given graphical
presentation for $H(S)$ is $C'(1/6)$.  If $h$ represents the identity
element then each path in the Cayley graph that follows the word
defining $h$ is a closed loop.  In this case by the main theorem of
graphical small cancellation, Theorem~\ref{thm:mainofsc}, this loop
must contain more than half of some cycle in one of the relator
graphs.  By condition~5 of Definition~\ref{defn:kdefn}, the perimeter
of each polygon is more than twice the girth~$g$ and so $h$ cannot
contain more than half of the degree $n$ subdivision of $\partial P$
for $P$ any polygon of $K$.  The given word for $h$ also contains only
pieces consisting of less than $1/6$ of the length of any cycle in a
defining relator of $H(S)$ of degree $m\neq n$, since any such piece
is of the form $a_i^la_{i+1}^l$ for $l$ equal to the minimum of
$|m|$~and~$|n|$, and $g\geq 13$.  

It remains to show that the kernel $K_{S,T}$ is acyclic.  Since
$E(S)$ is contractible and $K_{S,T}$ acts freely cellularly on
$E(S)$, it follows that $E(S)/K_{S,T}$ is an Eilenberg--Mac~Lane
space for $K_{S,T}$.  By Proposition~\ref{prop:mainarg}, $E(S)/K_{S,T}$
has the same homology as $F$, which is contractible because it has
the same homotopy type as $E(T)$.  
\end{proof}

\section{Graphs of groups} 
\label{sec:graphgp}

In this section we prove Proposition~\ref{prop:main} and 
Corollary~\ref{cor:main}.  We start by proving two general
results that we will use.

\begin{proposition}\label{prop:pullback} 
  Let $f:X\rightarrow Y$ be a map of Eilenberg--Mac~Lane
  spaces, and suppose that the induced map $f_*:G\rightarrow Q$
  of fundamental groups is surjective, with kernel $N$.
  Let $\pi:\widetilde{Y}\rightarrow Y$ be the universal
  covering of $Y$, and let $P$ be the pullback of this
  covering along $f$.  Then $P$ is an Eilenberg--Mac~Lane
  space for $N$.
\end{proposition}

\begin{proof}
  Recall that the pullback $P$ is defined by 
  \[P=\{(x,y)\in X\times \widetilde{Y}\,\,:\,\, f(x)=\pi(y)\}.\]
  For any $f:X\rightarrow Y$, this is a covering space of $X$
  that is regular and has $Q$ as a group of deck transformations,
  where the action of $q\in Q$ is defined by $q(x,y)=(x,qy)$ and 
  the covering map is the map $(x,y)\mapsto x$.  The action of
  $Q$ is transitive on each fibre of this map.  In the case
  when $f_*:G\rightarrow Q$ is \emph{surjective} it may be shown
  that $P$ is connected.  To see this, if $\pi(y)=\pi(y')=f(x)$,
  then there is a loop $\beta$ in $Y$ that lifts to a path in $\widetilde{Y}$
  from $y$ to $y'$.  Now let $\gamma$ be a loop in $X$ based at $x$
  so that $f\circ \gamma$ is in the same homotopy class as $\beta$.
  The loop $\gamma$ and a lift to $\widetilde{Y}$ of the loop
  $f\circ\gamma$ together define a path in $P$ from $(x,y)$ to
  $(x,y')$.  Since $P$ is connected and $Q$ acts transitively on
  each fibre of the map $P\rightarrow X$, it follows that in
  this case $Q$ is the whole group of deck transformations.  

  Up to isomorphism there is only one connected regular covering of
  $X$ with $Q$ as its group of deck transformations: the space
  $\widetilde{X}/N$, where $\widetilde{X}$ is the universal covering
  of $X$.  This space is an Eilenberg--Mac~Lane space for $N$ and
  we have seen that it is homeomorphic to $P$.
\end{proof}

\begin{proposition}\label{prop:kergraphofgroups} 
  For $i\in \{1,\ldots,m\}$ let $H_i$ be a subgroup of a group $H$
  and let $N$ be a normal subgroup of $H$ so that $N\cap H_i$ is
  trivial for each $i$.  Let $G_i$ be another group that contains
  $H_i$ as a subgroup.  Let $G$ be the fundamental group of the
  star-shaped graph of groups with $m$ arms where the central
  vertex group is $H$, the edge groups are $H_1,\ldots, H_m$
  and the outer vertex groups are $G_1,\ldots,G_m$.  Let $\overline{G}$
  be the fundamental group of a similar graph of groups in which
  the central vertex group is replaced by $\overline{H}=H/N$.
  The quotient map $H\rightarrow H/N$ together with the identity
  maps on the edge groups and outer vertex groups induces a
  surjective group homomorphism $G\rightarrow \overline{G}$
  and the kernel of this homomorphism is isomorphic to a free
  product of copies of $N$, where the copies are indexed by
  the cosets of $\overline{H}$ in $\overline{G}$.
\end{proposition}

\begin{proof}
  We use the previous proposition to construct an Eilenberg--Mac~Lane
  space for the kernel of the map $G\rightarrow \overline{G}$ from
  which the stated result will be apparent.  Let $\overline{X}$ be
  an Eilenberg--Mac~Lane space for $\overline{G}$, let $X$ be an
  Eilenberg--Mac~Lane space for $G$, and let $f:X\rightarrow \overline{X}$
  be a based map that realizes the quotient map $G\rightarrow \overline{G}$.
  Let $Y_i$ be an Eilenberg--Mac~Lane space for $H_i$, let $Z_i$ be an
  Eilenberg--Mac~Lane space for $G_i$ and let $f_i:Y_i\rightarrow X$
  and $g_i:Y_i\rightarrow Z_i$ be based maps so that ${f_i}_*:H_i\rightarrow H$
  and ${g_i}_*:H_i\rightarrow G_i$ are the inclusions.  Define $\overline{f_i}=
  f\circ f_i$ and define $\overline{g_i}=g_i$.

  We can use the spaces $X$, $Y_i$, $Z_i$ together with the maps
  $f_i$, $g_i$ to make a star-shaped graph of spaces and we can use
  the spaces $\overline{X}$, $Y_i$, $Z_i$ together with the maps
  $\overline{f_i}$, $\overline{g_i}$ to make a second star-shaped
  graph of spaces.  By part~2 of Proposition~\ref{prop:graphofgroups},
  these spaces are Eilenberg--Mac~Lane spaces for $G$ and
  $\overline{G}$ respectively.  Moreover, we may define a map of
  graphs of spaces by taking the map $f:X\rightarrow \overline{X}$ on
  the central vertex space and the identity map on each edge space
  $Y_i$ and on each outer vertex space $Z_i$.  This gives us an
  explicit map of Eilenberg--Mac~Lane spaces inducing the surjection
  $G\rightarrow \overline{G}$ on fundamental groups.  By
  Proposition~\ref{prop:pullback}, the pullback of the universal
  covering space for the space for $\overline{G}$ along this map is an
  Eilenberg--Mac~Lane space for the kernel.

  The universal covering space of the graph of spaces that is an
  Eilenberg--Mac~Lane space for $\overline{G}$ is well understood:
  it can be viewed as another star-shaped graph of spaces, where
  each of the spaces arising is a disjoint union of copies of the
  universal cover of the original space, so that for example over
  the edge $I\times Y_i$ lies a disjoint union of copies of $I\times
  \widetilde{Y_i}$.  This space can also be understood as a graph of
  spaces in a different way, with each vertex and edge space being a
  single component of the union described above.  In this way, the
  universal covering is described as a graph of contractible spaces.
  Since it is by definition simply-connected, the graph underlying
  this graph of spaces is a tree: this is the Scott--Wall approach
  to constructing the Bass--Serre tree for $\overline{G}$ as a
  graph of groups~\cite[1.B]{hatcher}~or~\cite{scottwall}.

  The pullback space can also be understood as a graph of groups
  with underlying graph the Bass--Serre tree for $\overline{G}$
  expressed as a star-shaped graph of groups.  In this case,
  every edge space is contractible, and the vertex spaces that
  correspond to leaf vertices of the star-shaped graph are also
  contractible, while the vertex spaces over vertices that map
  to the central vertex of the star are Eilenberg--Mac~Lane
  spaces for $N$.  Thus the kernel of the group homomorphism
  is the free product of copies of $N$ indexed by the vertices
  of the Bass--Serre tree for $\overline{G}$ that map to the
  central vertex of the star, or equivalently indexed by the
  cosets $\overline{G}/\overline{H}$.
\end{proof} 

\begin{proposition} For each $S\subseteq T\subseteq Z$, the kernel of
  the homomorphism $G(S)\rightarrow G(T)$ is acyclic.
\end{proposition}

\begin{proof}
  Both $G(S)$ and $G(T)$ are constructed as star-shaped graphs of
  groups, with edges indexed by the polygons of $K$.  The leaf vertex
  groups are $G_P$ in each case and the edge groups are $H_P$ in each
  case.  The only difference is that the central vertex group is
  $H(S)$ for constructing $G(S)$ and $H(T)$ for constructing $G(T)$.
  Thus the hypotheses of
  Proposition~\ref{prop:kergraphofgroups} are satisfied, and we deduce
  that the kernel of the surjection $G(S)\rightarrow G(T)$ is
  isomorphic to the free product of copies of $K_{S,T}$ and so is
  itself acyclic.
\end{proof}

\begin{proposition} 
There is a finite Eilenberg--Mac~Lane space for 
the group $G(\emptyset)$ as in Proposition~\ref{prop:main}. 
If each $G_P$ has a finite $2$-dimensional Eilenberg--Mac~Lane
space then so does $G(\emptyset)$.  
\end{proposition} 

\begin{proof} 
Let $X$ be the presentation 2-complex for $H(\emptyset)$, 
which is built from standard cells, and has finite 1-skeleton.  
For each polygon $P$, let $Y_P$ be the presentation 2-complex 
for $H_P$, and let $Z_P$ be a finite Eilenberg--Mac~Lane space
for $G_P$.  Since the generators and relations for each $H_P$ 
are a subset of those of $H(\emptyset)$, the inclusion 
$H_P\rightarrow H(\emptyset)$ is induced by an isomorphism 
between $Y_P$ and a subcomplex of $X$.  Note also that each 
2-cell of $X$ is contained in exactly one of these subcomplexes.  

If $f_P:Y_P\rightarrow Z_P$ is a map that induces the embedding
$H_P\rightarrow G_P$, then an Eilenberg--Mac~Lane space $W$ for
$G(\emptyset)$ can be built by taking a copy of the mapping cylinder
of $f_P:Y_P\rightarrow Z_P$ for each $P$, and identifying the copy of
$Y_P$ with its image in $X$.  If $Y_P^1$ denotes the finite 1-skeleton
of $Y_P$, then the mapping cylinder of the restriction $f_P|_{Y_P^1}$
is a deformation retract of the mapping cylinder for $f_P$, and is a
finite complex.  Since each 2-cell of $X$ belongs to a unique polygon
$P$, these deformation retractions can be combined.  This shows that
the finite complex $W'$ obtained from $X^1$ and the mapping cylinders
of the maps
$f_P|_{Y_P^1}$ by identifying each copy of $Y_P^1$ with its image in
$X^1$ is a deformation retract of $W$.  If each $Z_P$ is 2-dimensional
then so is $W'$.  
\end{proof}

\begin{proof} (of Corollary~\ref{cor:main}) 
  Recall
  from~\cite[Sec.~15]{fpg} the set-valued invariant
  $\calr(\bg,G)\subseteq \zz$ for a group $G$ and a sequence
  $\bg=(g_1,\ldots,g_l)$ of elements of $G$, defined by
  $$\calr(\bg,G)=\{n\in \zz\,:\, g_1^ng_2^n\cdots g_l^n=1\}.$$
  In~\cite[Prop.~15.2]{fpg}, three properties of this invariant
  were established: for a fixed isomorphism type of countable group $G$,
  the invariant takes only countably many values as $\bg$ varies;
  if $H\geq G$ then $\calr(\bg,G)=\calr(\bg,H)$; if $G$ is finitely
  presented then $\calr(\bg,G)$ is recursively enumerable.

  Let $a_1,\ldots,a_g$ be a directed loop in $K$ of length equal
  to the girth of $K^1$.  Viewing this loop as a sequence of elements
  of $H(S)$, the Dehn property implies that for any $S\subseteq Z$,
  $a_1^n\cdots a_g^n=1$ in $H(S)$ if and only if $n\in S$.  Hence
  for any $S\subseteq Z$ one has 
  $$\calr((a_1,\ldots,a_g),G(S))=\calr((a_1,\ldots,a_g),H(S))=S\cup \{0\}.$$
  From this together with the known properties of the invariant
  $\calr$ it follows that there are continuously many isomorphism
  types of groups $G(S)$, that $G(S)$ is finitely presented if and
  only if $S$ is finite,   and that $G(S)$ can embed in a finitely
  presented group only when $S$ is recursively enumerable.  For
  the converse, note that since we know that $H_P$ embeds in a
  finitely presented group (see the remark at the end of the next
  section) it follows that $Z$ is recursively
  enumerable.  Now if $S$ and $Z$ are both recursively enumerable
  $G(S)$ is recursively presented and so by the Higman embedding
  theorem~\cite{hig,lynsch}, $G(S)$ does embed in some finitely
  presented group.  

  We know already that $G(\emptyset)$ has geometric dimension two and
  is of type $F$; since the kernel of the map $G(\emptyset)\rightarrow
  G(S)$ is acyclic it follows from Proposition~\ref{prop:coh} that
  $G(S)$ is of type $FH$ and has cohomological dimension two.  
  \end{proof}

\section{Embedding polygon subgroups}
\label{sec:embed}

In this section we construct embeddings of the polygon subgroup 
$H_P$, whose isomorphism type depends only on the perimeter $l$ 
of $P$ and on $Z\subseteq \zz-\{0\}$, into groups $G_P$ of type~$F$.  
Moreover, each group $G_P$ will have geometric dimension two.

The first case that we deal with is the case $Z=\zz-\{0\}$.  Before
starting this case we recall that the right-angled Artin group $A_L$ 
associated to a flag simplicial complex $L$ is the group with generators 
the vertices of $L$, subject to the relations that the two vertices 
incident on each edge commute.  The right-angled Artin group associated
to a finite flag complex $L$ is of type~$F$ and has cohomological dimension
one more than the dimension of $L$~\cite{kimroush,davbook}.  
The Bestvina--Brady group $BB_L$ is defined to be 
the kernel of the map $A_L\rightarrow \zz$ that sends each of the 
generators to $1\in \zz$.  Provided that $L$ is connected, there 
is a generating set for $BB_L$ that corresponds to the directed 
edges of $L$, where the edge from vertex $x$ to $y$ corresponds to the 
element $xy^{-1}$ in $A_L$.  A presentation for $BB_L$ in terms of 
these generators is given in~\cite{dicksleary}.  
In the case when $Z=\zz-\{0\}$, the given presentation for $H_P$ 
is equal to this presentation for $BB_{\partial P}$.  Hence we may 
take for $G_P$ the right-angled Artin group $A_{\partial P}$.  
However, it may be more helpful to use the natural isomorphism 
between the right-angled Artin group $A_L$ and the Bestvina--Brady 
group $BB_{C(L)}$ for the cone on $L$: if $c$ is the cone vertex, 
this isomorphism takes the vertex generator $x$ to the generator 
$xc^{-1}$ corresponding to the edge from $x$ to $c$.  

\begin{proposition} 
Let $K$ be a spectacular $2$-complex, let $Z=\zz-\{0\}$ and let $L$ be
the flag complex obtained from $K$ by replacing each polygon with the
cone on its boundary.  With the embedding $H_P\rightarrow G_P$ as
described above there is an isomorphism, for each $S\subseteq
\zz-\{0\}$, from $G(S)$ to the generalized Bestvina--Brady group
$G_L(S\cup\{0\})$ in the sense of~\cite[Defn.~1.1]{fpg}.
\end{proposition} 

\begin{proof}
  The generating set for $H(S)$
  consists of the directed edges of $K^1$, which is a subcomplex of
  $L$, the generating set for each $H_P$ is identified with the edges
  of $\partial P$, and the generating set for $G_P$ is identified with
  the edges from vertices of $\partial P$ to the cone vertex $c_P$.
  This gives a generating set for the group $G(S)$ consisting of
  directed edges of $L$.  Since each edge of $L$ is either in the
  image of $K^1$ or is incident on some cone vertex, this generating
  set for $G(S)$ consists of all of the directed edges of $L$.  But
  the generating set for the presentation for $G_L(S\cup \{0\})$ given
  in~\cite[Defn.~1.1]{fpg} is also the directed edges of~$L$.  Hence
  there are natural mutually inverse bijections between the generators
  of $G(S)$ and the generators of $G_L(S\cup\{0\})$.

  It remains to show
  that these bijections of generating sets send the relators of each
  group to valid relations in the other group, where the relators
  taken for $G(S)$ are those implicit in its description as a graph
  of groups.  To do this one employs the sort of reasoning that was
  used in the proof of~\cite[Prop.~2]{dicksleary}.
  The relations of the given presentation for $G(S)$ are: 
  \begin{enumerate}
  \item{} For each triangle in $L$ with directed boundary $(a,b,c)$
    the relators $abc$ and $a^{-1}b^{-1}c^{-1}$;

  \item{} For each polygon $P$ of $K$ with directed boundary
    $(e_1,\ldots,e_l)$ and each $n\in \zz-S$ the relator
    $e_1^ne_2^n\cdots e_l^n$;

  \item{} For each directed cycle $(e_1,\ldots,e_l)$ in $K^1$ and
    each $n\in S$ the relator $e_1^ne_2^n\cdots e_l^n$.

  \end{enumerate}
  On the other hand, the relations in the given
  presentation for $G_L(S\cup\{0\})$ are: 

  \begin{enumerate}
  \item{} For each triangle in $L$ with directed boundary $(a,b,c)$
    the relators $abc$ and $a^{-1}b^{-1}c^{-1}$;

  \item{} For each directed cycle $(e_1,\ldots,e_l)$ in $L$ and
    each $n\in S$ the relator $e_1^ne_2^n\cdots e_l^n$.

  \end{enumerate}

  The relators that are not common to the two presentations are the
  relators of the second type in the presentation for $G(S)$ 
  and the relators of the second type in the presentation for
  $G_L(S\cup\{0\})$ that correspond to cycles in $L$ not contained in
  $K^1$.
  
  The relators of the second type in the presentation for $G(S)$
  associated to the boundary of a given polygon $P$ can be deduced
  from the triangle relators for the triangles that form the
  conical subdivision of $P$, as in the proof of~\cite[Prop.~2]{dicksleary}.

  To deduce the relators of the second type in the presentation for
  $G_L(S\cup\{0\})$ from the relators of the presentation for $G(S)$,
  one again reasons as in the proof of~\cite[Prop.~2]{dicksleary}. 
  Implicit in that proof is the statement that if $(a_1,\ldots,a_l)$
  and $(b_1,\ldots,b_m)$ are directed cycles that are (unbased) homotopic
  to each other in $L$, then for any $n$, the relation $a_1^n\cdots a_l^n$ is a
  consequence of the relator $b_1^n\cdots b_m^n$ together with the
  triangle relators.  This suffices, since any edge cycle in $L$ is homotopic to
  an edge cycle that is contained in $K^1$.
\end{proof} 

Next we consider the case $Z=\{k^n\,:\,n\geq 0\}$ for some $k\in \zz$
with $|k|>1$.  In this case the group $H_P$ has an injective but
non-surjective self-homomorphism $\phi=\phi_P$ defined by
$\phi(a_i):=a_i^k$ for each of the edge generators $a_i$.  In this 
case the natural choice for $G_P$ is the ascending HNN-extension
$$G_P=\langle a_1,\ldots,a_l,t\,: \,a_i^t=a_i^k,\,a_1a_2\cdots a_l=1\rangle,$$
in which conjugation by the stable letter $t=t_P$ acts by applying 
the homomorphism $\phi$.  

\begin{proposition} \label{prop:kpowers} 
The presentation $2$-complex for the finite presentation for $G_P$ given 
above is aspherical. 
\end{proposition} 

\begin{proof} 
Let $Y=Y_P$ be the presentation 2-complex for the small cancellation
presentation for $H_P$, and let $f:Y\rightarrow Y$ be the based
cellular map that induces $\phi:H_P\rightarrow H_P$ on fundamental
groups.  The natural choice of Eilenberg--Mac~Lane space for $G_P$ is
the mapping torus $M=M_f$ of $f$.  Since $f$ sends relators to
relators, there is an easy way to put a CW-structure on $M$, with
finite 1-skeleton.  Each $i$-cell of $Y$ contributes one $i$-cell and
one $(i+1)$-cell to $M$.  Hence the cells of $M$ are: one 0-cell,
$l+1$ 1-cells labelled by the generators $a_1,\ldots,a_l,t$, one
family of $l$ trapezoidal 2-cells (coming from the 1-cells of $Y$)
whose boundaries are the words $ta_it^{-1}a_i^{-k}$, an infinite
family of 2-cells and an infinite family of 3-cells.  For $n\geq 0$,
denote by $e_n$ the 2-cell that corresponds to the relator
$a_1^{k^n}a_2^{k^n}\cdots a_l^{k^n}$, and for $n\geq 1$ let $E_n$
be the 3-cell coming from the 2-cell $e_{n-1}$, so that the boundary
of $E_n$ consists of $e_{n-1}$, 
$-e_n$, and $k^{n-1}$ copies of each of the trapezoidal 2-cells.  For
$n\geq 0$, let $M_n$ be the subcomplex of $M$ that contains the
1-skeleton, the trapezoidal 2-cells, the 2-cells $e_0,\ldots,e_n$ and
the 3-cells $E_1,\ldots,E_n$.  There is a deformation retraction of
$E_n$ onto $\partial E_n-\hbox{Int}(e_n)$, and combining this with the
identity map on the rest of $M_{n-1}$ defines a deformation retraction
of $M_n$ onto $M_{n-1}$.  Applying these retractions successively so
that the $n$th of the retractions happens during the interval
$[1/2^n,1/2^{n-1}]$ gives a deformation retraction of $M$ onto $M_0$.
Since $M_0$ is the presentation 2-complex described in the statement
this implies that $M_0$ is aspherical.
\end{proof}

Sapir has given an aspherical version of the Higman
embedding theorem, stating that any finitely generated group with
an aspherical recursive presentation can be embedded into a
group with a finite aspherical presentation~\cite{sapir}.
This gives a characterization of which polygon groups embed
into groups of type~$F$.  

\begin{proposition}
  Suppose that $P$ is a polygon of perimeter at least~$13$, and
  let $H_P$ be the corresponding polygon group, which depends on
  $Z\subseteq \zz-\{0\}$ as well as on $P$.  The
  following statements are equivalent.
  \begin{itemize}
  \item{} The set $Z$ is recursively enumerable;

  \item{} $H_P$ embeds in a finitely presented group;

  \item{} $H_P$ embeds in a group admitting a finite $2$-dimensional
    Eilenberg--Mac~Lane space.
  \end{itemize}
\end{proposition}

\begin{proof}
  Since the perimeter of $P$ is at least 13, the defining presentation
  for $H_P$ satisfies the $C'(1/6)$ small cancellation condition.
  Given this, if $a_1,\ldots,a_l$ are the directed edges making up the
  bounding cycle of $P$, it follows from the Dehn property
  that $a_1^na_2^n\cdots a_l^n=1$ in $H_P$ if and only if $n\in Z\cup\{0\}$.
  Since any finitely generated subgroup of a finitely presented
  group is recursively presented, it follows that if 
  $H_P$ embeds into a finitely presented
  group then $Z$ must be recursively enumerable.

  In general, the $C'(1/6)$ condition implies that the presentation 2-complex
  for $H_P$ is aspherical.  If $Z$ is recursively enumerable, then $H_P$ is
  recursively presented and Sapir's
  theorem~\cite{sapir} implies that $H_P$ embeds in a group admitting a
  finite 2-dimensional Eilenberg--Mac~Lane space.  Such a group is
  a fortiori finitely presented.
\end{proof}

\section{The projective line and its symmetries}
\label{sec:pgltwo}

Our construction of a spectacular 2-complex will involve the combinatorics
of the 2-dimensional projective linear group over a finite field, viewed
as a group of permutations of the projective line.  For the benefit of
the reader we summarize those properties that we shall use.  Recall that
a permutation group on a set $X$ is said to be \emph{$k$-transitive} if
it acts transitively on the ordered $k$-tuples of elements of $X$, and 
\emph{strictly $k$-transitive} if in addition the stabilizer of an
ordered $k$-tuple is trivial.

\begin{theorem}\label{thm:threetrans}
  Let $F$ be any field and 
  let $G=PGL(2,F)$ be the $2$-dimensional projective general linear group
  over $F$.  The action of $G$ on the projective
  line over $F$ is strictly $3$-transitive.
\end{theorem}

\begin{proof}
  Let $\be_1,\be_2$ be the standard basis for the vector space $F^2$,
  viewed as a space of column vectors, and let
  \[M=
  \begin{pmatrix}
    a&b \\
    c&d
  \end{pmatrix},\quad ad\neq bc\]
  be an element of $GL(2,F)$.  Then $M\be_1=a\be_1+c\be_2$, and since $a,c$
  can be arbitrary provided that they are not both zero, $GL(2,F)$ acts
  transitively on the non-zero vectors.  The stabilizer of the line
  $\langle \be_1\rangle$ is the subgroup consisting of those $M$ with
  $c=0$.  For such $M$, $ad\neq 0$ and $b$ is arbitrary (independent of
  choice of $a,d$).  For such $M$, $M\be_2=b\be_1+d\be_2$, an arbitrary
  vector in $F^2-\langle \be_1\rangle$.  It follows that $GL(2,F)$ acts
  2-transitively on the non-zero vectors in $F^2$.  The intersection of
  the stabilizers of $\langle \be_1\rangle$ and $\langle \be_2\rangle$
  is those $M$ with $b=c=0$; this implies that $ad\neq 0$.  For such
  $M$, we have that $M(\be_1+\be_2)=a\be_1+d\be_2$, an arbitrary element
  of $F^2-(\langle \be_1\rangle \cup \langle \be_2\rangle)$.  Hence
  $GL(2,F)$ acts 3-transitively on the lines in $F^2$.  The vector
  $a\be_1+d\be_2$ is in the line $\langle \be_1+\be_2\rangle$ if and
  only if $a=d$.  Thus the intersection of the stabilizers (in $GL(2,F)$)
  of the lines $\langle \be_1\rangle$, $\langle \be_2\rangle$ and
  $\langle \be_1+\be_2\rangle$ is the scalar matrices $aI$.  Since these
  form the kernel of the map $GL(2,F)\rightarrow G$ it follows that $G$
  acts strictly 3-transitively as claimed.
  \end{proof}

  \begin{theorem}\label{thm:normalizers}
    Fix a prime power $q=p^k$, let $G=PGL(2,q)$ be the $2$-dimensional
    projective general linear group over the field with $q$ elements,
    and let $\pp^1(q)$ denote the projective line over the field of $q$
    elements.  
    The following statements hold:
    \begin{itemize} 
  \item{}
  The orders of elements of $G$ are $p$, the prime dividing $q$, and every
  factor of $q+\epsilon$ for $\epsilon=\pm 1$.

\item{}
  Let $g$ be an element of $G$ of order $d\geq 3$ with $d$ dividing
  $q+\epsilon=q\pm 1$.

  \begin{itemize}
  \item{}
    The centralizer of $g$ is cyclic of order $q+\epsilon$.

  \item{}
    The normalizer of the subgroup generated by $g$ is dihedral of
    order $2(q+\epsilon)$.

  \item{}
    In its permutation action on $\pp^1(q)$, $g$ fixes
    $1-\epsilon$ points and permutes all other points in
    $(q+\epsilon)/d$ cycles of length $d$.
  \end{itemize}

\item{}
  Any element of order $p$
  acts on $\pp^1(q)$ with one fixed point and $q/p$ cycles of length~$p$.

\item{}
  When $q$ is odd, there are two cycle types of elements of order two;
  for each $\epsilon =\pm 1$ there are elements of order two whose centralizers
  are dihedral of order $2(q+\epsilon)$.  These elements fix
  $1-\epsilon$ points of $\pp^1(q)$ and have $(q+\epsilon)/2$ cycles of
  length two.  
  
  \end{itemize}
\end{theorem}

  \begin{proof}
    We start with some general remarks concerning normalizers inside
    permutation groups.  
    If $g$ is an element of a permutation group $G$ on a set $X$, and
    $m$ is any integer, then the set of points fixed by $g^m$ will
    contain the set of points fixed by $g$.  If $g$ and $g^m$ generate
    the same cyclic subgroup of $G$, then there is an integer $m'$ so
    that $g=(g^m)^{m'}$, and so by symmetry $g$ and $g^m$ must fix the
    same points.  It follows that any $h\in G$ that normalizes the
    subgroup generated by $g$ must preserve the set of $g$-fixed points.

    Now consider the case of interest, when $G=PGL(2,q)$ acting as a group of
    permutations of $\pp^1(q)$.  
    By Theorem~\ref{thm:threetrans}, any non-identity element
    of $G$ fixes at most two points.  Lifting to the general linear group,
    it follows that any non-scalar matrix in $GL(2,q)$ can fix at most two
    lines in $\ff_q^2$.  

    There are three possible cases for the characteristic polynomial of an
    element of $GL(2,q)$: it may factor as the square of a linear polynomial
    $(t-\lambda)^2$, it may factor as the product of two distinct linear 
    polynomials $(t-\lambda)(t-\mu)$ or it may be a quadratic polynomial
    that has no roots in $\ff_q$.  Any eigenvalues (i.e., roots of the
    characteristic polynomial) must be non-zero.  

    A matrix whose characteristic polynomial is $(t-\lambda)^2$ is
    either a scalar matrix $\diag(\lambda,\lambda)$ or it is conjugate
    to an upper triangular matrix with both diagonal entries equal
    to~$\lambda$.  A scalar matrix is in the kernel of the map
    $GL(2,q)\rightarrow G$.  The set of all upper triangular matrices
    with 1's on their diagonal forms a subgroup of $GL(2,q)$
    isomorphic to the additive group of $\ff_q$; in particular each of
    these except the identity matrix has order $p$, and maps to a
    non-identity element of $G$ which must also have order $p$.  An
    upper triangular matrix with both diagonal entries equal to
    $\lambda$ is the product of a matrix as considered above with a
    scalar matrix.  Hence its image in $G$ is also an element of
    order~$p$.  To see that the cycle type of an element of order $p$
    acting on $\pp^1(q)$ is as claimed, note that $|\pp^1(q)|=q+1$ is
    congruent to 1 modulo~$p$, and so the number of fixed points under
    the action of an element of order $p$ must also be congruent to 1
    modulo~$p$ and cannot be larger than~2 by strict 3-transitivity.

    A matrix whose characteristic polynomial is $(t-\lambda)(t-\mu)$
    for $\lambda\neq \mu$ is conjugate to the diagonal matrix
    $\diag(\lambda,\mu)$ via the change of basis that sends $\be_1$ to
    an eigenvector for $\lambda$ and $\be_2$ to an eigenvector for
    $\mu$.  Let $g$ be the image of $\diag(\lambda,\mu)$ inside $G$.
    Since $\lambda\neq \mu$, $g$ is not the identity element and so
    $g$ fixes just two points of $\pp^1(q)$: the lines
    $\langle \be_1\rangle$ and $\langle \be_2\rangle$.  Any element
    of $G$ that normalizes the subgroup generated by $g$ must
    either fix or exchange these two lines.  The matrices that preserve
    the lines $\langle \be_1\rangle$ and $\langle \be_2\rangle$ are
    precisely the diagonal matrices, each of which centralizes
    $\diag(\lambda,\mu)$, and the matrices that exchange these
    two lines are precisely the antidiagonal matrices.  Each
    antidiagonal matrix conjugates $\diag(\lambda,\mu)$ to
    $\diag(\mu,\lambda)$.  Since the product of $\diag(\lambda,\mu)$
    and $\diag(\mu,\lambda)$ is a scalar matrix, any $h\in G$ that
    exchanges the two lines must conjugate $g$ to $g^{-1}$.
    Thus the normalizer in $G$ of the subgroup generated by $g$ is
    the image in $G$ of the group of diagonal and antidiagonal
    matrices in $GL(2,q)$.  This is a dihedral group of order
    $2(q-1)$.  The cyclic subgroup of order $q-1$ fixes the same
    two points of $\pp^1(q)$ as $g$ and centralizes $g$.
    The other elements of this dihedral group swap the two points
    fixed by $g$ and send $g$ to its inverse.
    
    In this case, it remains only to consider the cycle type for the
    action of $g$.  Since the pointwise stabilizer in $G$ of the two
    points fixed by $g$ is a cyclic group of order $q-1$ containing
    $g$, it follows that $g$ is a power of an element of this order.  
    If $g$ has order $d>1$ then it is the $(q-1)/d$th power of some
    element of order $q-1$. To determine the cycle type of $g$  it
    suffices to show that this element of order $q-1$
    acts on $\pp^1(q)$ via a single $(q-1)$-cycle.  
    A matrix in $GL(2,q)$ that maps to this element is
    conjugate to $\diag(a,d)$ where $a/d$ is a
    generator for the multiplicative group of $\ff_q$.  For
    $0\leq k\leq q-2$ the vectors $a^k\be_1+d^k\be_2$
    all lie in distinct lines, showing that such an element 
    acts as a $(q-1)$-cycle.

    An irreducible quadratic polynomial over $\ff_q$ has both of its
    roots in the field $\ff_{q^2}$ of $q^2$ elements.  We may view
    $GL(2,q)$ as a subgroup of $GL(2,q^2)$, and we may view $G$ as a
    subgroup of $PGL(2,q^2)$.  If $M$ is a matrix in $GL(2,q)$ whose
    characteristic polynomial is irreducible over $\ff_q$, then over
    $\ff_{q^2}$ we see that $M$ is conjugate to a diagonal matrix
    $\diag(\lambda,\mu)$ for some $\mu\neq \lambda\in \ff_{q^2}$.  The
    study of the properties of diagonalizable matrices in the previous
    two paragraphs therefore applies to $M$ as an element of
    $GL(2,q^2)$.  If $g$ is the image of the matrix $M$ in
    $PGL(2,q^2)$, then we see that the normalizer in $PGL(2,q^2)$ of
    this element is dihedral of order $2(q^2-1)$, and that the
    elements of this subgroup that fix the two points of $\pp^1(q^2)$
    that are fixed by $g$ form a cyclic subgroup of order $q^2-1$,
    while the elements that swap these two points conjugate $g$ to
    $g^{-1}$.  Arguing as in the previous paragraph, we see that a
    generator for this cyclic group acts as a single $(q^2-1)$-cycle
    on the complement of the two fixed points.  Hence the points of
    $\pp^1(q^2)$ that are not fixed by $g$ lie in $g$-orbits of length
    equal to the order of $g$.  The action of $g$ on $\pp^1(q^2)$
    preserves the subset $\pp^1(q)\subseteq \pp^1(q^2)$.  Since the
    two points of $\pp^1(q^2)$ that are fixed by $g$ are not contained
    in $\pp^1(q)$, the orbits for $g$ on $\pp^1(q)$ are all of length
    equal to the order of $g$.  Any matrix in $GL(2,q)$ that fixes the
    same two points of $\pp^1(q^2)$ as $g$ must also have an
    irreducible quadratic as its characteristic polynomial (since its
    eigenvectors in $\ff_{q^2}^2$ do not lie in $\ff_q^2$).  

    The roots of a quadratic that is irreducible over $\ff_q$ lie in a
    single orbit for the Galois group of $\ff_{q^2}$ over $\ff_q$.
    This Galois group is a group of order two generated by the
    Frobenius map $\lambda\mapsto \lambda^q$.  It follows that any
    matrix in $GL(2,q)$ whose characteristic polynomial is an
    irreducible quadratic is conjugate over $\ff_{q^2}$ to a matrix of
    the form $\diag(\lambda,\lambda^q)$ for some $\lambda\in
    \ff_{q^2}-\ff_q$.  Since $\lambda^{q^2}=\lambda$, the $(q+1)$st
    power of $\diag(\lambda,\lambda^q)$ is the scalar matrix
    $\diag(\lambda^{q+1},\lambda^{q+1})$.  Thus any element of
    $GL(2,q)$ whose characteristic polynomial is an irreducible
    quadratic gives rise to an element of $G=PGL(2,q)$ of order
    dividing $q+1$.  Combining this information with that given in
    the previous paragraph, we see that if $g$ is an element of
    $PGL(2,q)$ that fixes no point of $\pp^1(q)$, then the set of
    all elements of $PGL(2,q)$ that fix the same points in $\pp^1(q^2)$
    as $g$ is a cyclic group of order dividing $q+1$, and that this
    cyclic group has index at most two in the set of all elements of
    $PGL(2,q)$ that fix this \emph{set} of two points.  Furthermore,
    if $g$ has order $d$, then $g$ acts on $\pp^1(q)$ as $(q+1)/d$
    disjoint $d$-cycles.  

    So far we have an upper bound for the normalizer of a group
    element $g$ corresponding to a matrix $M$ whose characteristic
    polynomial is an irreducible quadratic, but we also need a lower
    bound.  I.e., we need to construct a cyclic group of order $q+1$
    that centralizes $g$ and a dihedral group of order $2(q+1)$ that
    normalizes the subgroup generated by $g$.  Let $f(x)\in \ff_q[x]$
    be the characteristic polynomial of such an $M$.  Define a (unital)
    ring homomorphism $\widetilde{\psi}_M:\ff_q[x]\rightarrow
    M_2(\ff_q)$ by $\widetilde{\psi}_M(x)=M$ and extending
    $\ff_q$-linearly.  By definition, the kernel of
    $\widetilde{\psi}_M$ is the ideal of $\ff_q[x]$ generated by the
    minimal polynomial of~$M$.  Since the characteristic polynomial
    $f(x)$ is an irreducible quadratic,
    it is equal to the minimal polynomial of $M$ and  
    the factor ring $E=\ff_q[x]/(f(x))$ is isomorphic to the field
    $\ff_{q^2}$.  Moreover, $\widetilde{\psi}_M$ induces an injective
    $\ff_q$-linear homomorphism $\psi_M:E\rightarrow M_2(\ff_q)$.
    Fix a non-zero vector $\bv$ in $\ff_q^2$ and define an $\ff_q$-vector space
    homomorphism $\theta_\bv:E\rightarrow \ff_q^2$ by
    $\theta_\bv(1)=\bv$ and $\theta_\bv(x)=M\bv$ where we have abused
    notation slightly by writing $x$ for the image of $x$ inside $E$.
    Since $E$ is being used as both a module and as a ring, we use
    the notation $E^1$ for $E$ viewed as a 1-dimensional vector space
    over $E$ and $M_1(E)$ for $E$ viewed as the endomorphism ring of $E^1$.  
    Since $M$ has no eigenvalues in $\ff_q$, $\bv$~and~$M\bv$ are linearly
    independent and so the map $\theta_\bv:E^1\rightarrow \ff_q^2$ is
    an isomorphism of vector spaces over $\ff_q$.   Moreover it is compatible
    with $\psi_M:M_1(E)\rightarrow M_2(\ff_q)$ in the sense that for
    all $\lambda,\mu\in E$, $\theta_\bv(\lambda\mu)=\psi_M(\lambda)
    \theta_\bv(\mu)$.  In particular,
    specializing to the case $\lambda=x$ gives that 
    $\theta_\bv(x\mu)=M\theta_\bv(\mu)$ for any $\mu$, and 
    so $\theta_\bv x\theta_\bv^{-1} = M$.  To avoid any confusion, we
    emphasize that this equation represents an identity between $\ff_q$-linear
    automorphisms of $\ff_q^2$, and that the symbol $x$ in this equation
    represents the map
    $\lambda\mapsto x\lambda$ for all $\lambda\in E^1$ which is $E$-linear
    and hence a fortiori $\ff_q$-linear.  

    The multiplicative group $GL(1,E)\leq M_1(E)$ is a cyclic group of
    order $q^2-1$ containing~$x$.  The Galois group of $E$ as an
    extension of $\ff_q$ is a cyclic group of order two, generated by
    the Frobenius map $\phi_q:\lambda\mapsto \lambda^q$.  This map is
    an $\ff_q$-linear automorphism of $E$ of order two which is not
    $E$-linear, so it normalizes $GL(1,E)$ but is not contained in
    $GL(1,E)$.  Since $\phi_q(\lambda\mu)=\lambda^q\mu^q$, it follows
    that conjugation by $\phi_q$ acts on $GL(1,E)$ as the group
    automorphism $\lambda\mapsto \lambda^q$.  Let $H$ be the group of
    $\ff_q$-linear automorphisms of $E^1$ generated by $GL(1,E)$ and
    $\phi_q$.  Thus $H$ is a group of $\ff_q$-linear automorphisms of
    $E$ that has order $2(q^2-1)$, that normalizes the subgroup
    generated by $x$, and contains an index two cyclic subgroup
    $GL(1,E)$ that contains $x$.  Now $\theta_\bv:E^1\rightarrow
    \ff_q^2$ is an isomorphism of $\ff_q$-vector spaces such that
    $\theta_\bv x \theta_\bv^{-1}=M\in M_2(\ff_q)$.  Hence the group
    $H'=\theta_\bv H\theta_\bv^{-1}\leq GL(2,q)$ normalizes the subgroup
    generated by the matrix
    $M$ and $M$ is contained in a cyclic subgroup $C$ of $H'$ of order
    $q^2-1$.  Since $M$ is a power of a generator for $C$ we see that
    each generator of $C$ has the same fixed
    points in $\pp^1(q^2)$ as $M$.  These points do not lie in $\pp^1(q)$
    and so each generator of $C$ is irreducible over $\ff_q$.  It follows
    by an argument given earlier that the $(q+1)$st power of any generator 
    of $C$ is a scalar matrix and so lies in the kernel of the map
    $GL(2,q)\rightarrow G$.  Hence the image of $C$ in $G$ is cyclic of
    order at most $q+1$.  Since this kernel of the map $GL(2,q)\rightarrow G$
    is cyclic of order $q-1$, the image of $C$ is cyclic of order 
    exactly $q+1$ and the image
    of $H'$ in $G$ is a dihedral group of order $2(q+1)$, with
    $g$ (the image of $M$) contained in a cyclic subgroup of order
    $q+1$ (the image of $C$).
    
    The statements concerning an element $g$ of order two follow from the
    above cases.  In particular, the case when $q=2^k$ is even corresponds
    to an element $g$ fixing one point of $\pp^1(q)$.  When $q$ is odd,
    the element $g$ must fix either 0~or~2 points of $\pp^1(q)$, with
    the other points permuted in 2-cycles.  Write $1-\epsilon$ with
    $\epsilon=\pm 1$ for the number of fixed points for $g$.  In each
    case, $g$ is a power of an element of order $q+\epsilon$, $g$ is
    centralized (or equivalently normalized) by a dihedral subgroup
    of order $2(q+\epsilon)$, and the cycle type of $g$ is as claimed.

    \end{proof}

\section{A spectacular 2-complex} 
\label{sec:complex}

We construct a 2-complex $K$ having the required properties 
in stages.  Initially we ignore the requirements of large
girth (or rotundity) and 
acyclicity and construct a 2-complex $K_1$ with perfect fundamental
group and the required small cancellation property.  By passing 
to a subcomplex we obtain an acyclic subcomplex $K_2$.  Finally, by
subdividing the 1-skeleton of $K_2$, with a corresponding increase 
in the number of sides of each polygon, we obtain $K$.

Fix a prime power $q$, and fix $d\geq 3$ so that $d$ divides
$q+\epsilon$ for some $\epsilon\in\{\pm 1\}$.  Now let $G=PGL(2,q)$ be
the 2-dimensional projective general linear group over the field with
$q$ elements.  There is a natural action of $G$ on the projective
line, a set of $q+1$ points.  The 1-skeleton $K_1^1$ is the complete
graph with vertex set the projective line.  By construction $G$ acts
on $K_1^1$, and the action is triply-transitive on the vertex set
$K_1^0$.  Note also that $|G|=q(q^2-1)$ is equal to the number of
ordered triples of distinct elements of $K_1^0$.

Let $g$ be an element of $G$ of order $d$.  Our complex will 
depend on the pair $(d,q)$ and on the conjugacy class of the
element $g$.  By Theorem~\ref{thm:normalizers}, the centralizer in 
$G$ of $g$ is cyclic of order $q+\epsilon$, and the normalizer
of the subgroup generated by $g$ is dihedral of order $2(q+\epsilon)$.
It follows that 
the conjugacy class of $g$ contains $q(q-\epsilon)$ elements and
that it is closed under taking inverses.  The element $g$ is 
a power of an element of order $q+\epsilon$. 
From this it follows that the permutation action of $g$ on $K_1^0$ has 
$1-\epsilon$ fixed points and $(q+\epsilon)/d$ cycles of 
length $d$.  We use each of these cycles of length~$d$ to
describe the attaching map for a $d$-gon.  If $x$ is an element
of the conjugacy class of $g$ and the vertices
$v_0,\ldots,v_{d-1},v_d$ are such that $v_i=x^i(v_0)$, then there
is a $d$-gon whose boundary is the edge loop of length~$d$
in $K_1^1$ consisting of the edges
$(v_0,v_1),(v_1,v_2),\ldots,(v_{d-1},v_d)$.  This edge loop 
passes through the vertices $v_0,\ldots,v_{d-1}$ consecutively.
Note that the vertex orbits for $x^{-1}$ give rise to the same
polygons as the vertex orbits for $x$, with the opposite
orientation.   Each pair of the form $x,x^{-1}$ of elements 
of the conjugacy class of $g$ gives a recipe for attaching 
$(q+\epsilon)/d$ distinct $d$-sided polygons to $K_1^1$.  
The complex $K_1$ is thus defined by attaching $|G|/2d$ distinct
$d$-gons to $K_1^1$.  

Complexes of this form in the case $d=q+1$ were studied by
Aschbacher and Segev in~\cite{aschseg}, although the complex
that they describe is actually the conical subdivision of our
complex, in which each $d$-gon is replaced by the cone consisting
of $d$ triangles.  We encourage the
reader to experiment with these complexes for small values of~$q$.  
There are four such complexes with $q\leq 4$: when $(d,q)=(3,2)$
the complex is a triangle; when $(d,q)=(4,3)$ the complex is a
projective plane made from three squares that can be understood
as the quotient of the boundary of a cube by its antipodal map;
when $(d,q)=(3,4)$ the complex is the 2-skeleton of a 4-simplex;
when $(d,q)=(5,4)$ the complex is the 2-skeleton 
of the Poincar\'e homology sphere~\cite[pp.~34--36]{thurston},
which is an acyclic space with fundamental group of order~120.

\begin{proposition}\label{prop:smallpieces} 
Each piece of the intersection of two polygons of $K_1$ contains 
at most one edge. 
\end{proposition} 

\begin{proof}
  Throughout this proof, the term \emph{triple} will mean an ordered
  triple of pairwise distinct vertices of $K_1$.  Say that a triple
  $(u,v,w)$ is \emph{contained} in a polygon $P$ if $v$ is a vertex
  of $P$ and the edges $\{u,v\}$ and $\{v,w\}$ are both contained in $P$. 
  Clearly each $d$-gon contains exactly $2d$ such triples.  
  It will suffice to show that no triple is contained
  in more than one polygon of $K_1$, because the vertices 
  contained in a piece of length~2 would form such a triple.  

  To do this, consider the set $\Pi$ of pairs consisting of a
  polygon $P$ of $K_1$ and a triple contained in the polygon $P$.
  We shall count the number of elements in $\Pi$ in two different
  ways, firstly by summing over the polygons and secondly by summing
  over the triples.  Before starting, note that 
  the group $G=PGL(2,q)$ acts on $K_1$ and hence also acts on $\Pi$.
  Since each polygon contains exactly $2d$ triples and there
  are $q(q^2-1)/2d$ polygons in $K_1$, we see that $|\Pi|=q(q^2-1)$,
  and so $\Pi$, the set of triples and $G$ all have the same
  number of elements.  Since the action of $G$ on the set of
  triples is free and transitive, it follows that each triple
  must be contained in the same number of polygons.  Hence
  $|\Pi|$ is equal to the number of triples times the number
  of polygons containing any given fixed triple.  Since $|\Pi|$
  is equal to the number of triples, it follows that 
  each triple is contained in exactly one polygon.  In
  particular, no triple is contained in more than one polygon.  
\end{proof} 

In the case when $d=q+1$, the Euler characteristic of $K_1$ is equal
to~1.  Aschbacher and Segev showed that many of
these complexes are rationally acyclic~\cite{aschseg}.  The only
Aschbacher--Segev complexes that are known to be (integrally) acyclic
are the two cases already mentioned above: $(d,q)=(3,2)$ and
$(d,q)=(5,4)$, i.e., the triangle and the 2-skeleton of the Poincar\'e
homology sphere.  We cannot use these complexes because we need $d\geq
7$ together with Proposition~\ref{prop:smallpieces} 
to ensure the $C'(1/6)$ condition
(condition~5 in Definition~\ref{defn:kdefn}).  

For $d< q+1$, the complex $K_1$ 
has too many 2-cells to be acyclic, but frequently $H_1(K_1)$ is 
trivial.  Consider the case $(d,q)=(7,8)$.  In this case $K_1$ 
has~36 2-cells, and a calculation shows that $H_1(K_1)=0$ and 
that $H_2(K_1)$ is free abelian of rank~8.  (There are three 
conjugacy classes of elements of order~7 in $G=PGL(2,8)$, but 
the action of the Galois group of the field of 8 elements 
by outer automorphisms of $G$ permutes the three classes, 
so there is only one isomorphism type of complex $K_1$.)  
For this $K_1$ there 
are many ways to remove eight 2-cells to leave an acyclic 
2-complex: randomly removing 2-cells to reduce the rank of 
$H_2$ often finds such a complex.  The simplest way to produce
such a complex that we have found is to fix one of the vertices
$v_0$ of $K_1$, and to discard the eight 2-cells that are \emph{not} 
incident on $v_0$.  

This 2-complex $K_2$ has all of the required properties
except that its polygons have too few sides and its girth is~3;
by subdividing each edge into five, 
one obtains a spectacular complex $K$ with girth~15 whose 28 
2-cells are 35-gons.

\section{Closing remarks} 
\label{sec:ques}

The groups $G(S)$ for $S\neq \emptyset$ are known to have
cohomological dimension~2, but for most of them we have been
unable to construct a 2-dimensional Eilenberg--Mac~Lane space.

In contrast to the Bestvina--Brady construction, graphical small
cancellation is purely 2-dimensional, and so our methods cannot
be used to construct (for example) groups that are finitely presented
but not of type $F$.

We have been unable to find a spectacular complex $K$ admitting an
action of a non-trivial finite group $Q$ so that the fixed point
set $K^Q$ is empty.  Such a pair might give an alternative
construction for the main examples in~\cite{vfg}.  

It is easy to see that the groups $H(S)$ fall into uncountably
many quasi-isometry classes using Bowditch's argument~\cite{bhb}.
It is less clear what happens with the groups $G(S)$, except in
the case when $G(S)$ isomorphic to a generalized Bestvina--Brady
group $G_L(S)$, which was covered in~\cite{kls}.  Since this
article was first submitted, the first named author has resolved
the case when $Z=\{k^n\,\,:\,\,n\geq 0\}$
and $G_P$ is as described in Proposition~\ref{prop:kpowers},
in~\cite{brownphd}.  The question of whether the groups $G(S)$
form uncountably many quasi-isomorphism classes for arbitrary
recursively enumerable $Z\subseteq \zz$ and arbitrary choices
of embeddings $H_p\rightarrow G_P$ remains open in full
generality.  The methods introduced in~\cite{mow} may prove
useful in resolving this but they  
do not seem to apply directly.  Fix a finite generating
set for $G(\emptyset)$, and use the image of this set under
the surjective homomorphism $G(\emptyset)\rightarrow G(S)$
as a set of generators for $G(S)$.  In the language 
of~\cite{mow}, the map $\psi:S\mapsto G(S)$ defines an injective
function from the set of subsets of $Z$ to the space of marked
groups.  To apply the main theorem
of~\cite{mow} it would suffice to show that the image of $\psi$
is closed, which in turn would follow if one could show that $\psi$
is continuous for the Tychonoff (or product) topology on the
powerset of $Z$.  

Condition~7 in the definition of a spectacular complex may be weakened,
to give a weaker conclusion.  For example if $K$ is only assumed to
be 1-acyclic (or equivalently to have perfect fundamental group),
then each $G(S)$ will be of type $FP_2$.  Similarly, if $K$ is only
assumed to be rationally acyclic, then each $G(S)$ will be of type
$FP(\qq)$.  Of course in these cases one could also weaken the
hypotheses on $G_P$.

It is hard to see how conditions 1--6 of Definition~\ref{defn:kdefn}
could be weakened.  Condition~3 is used only in the proof of
Proposition~\ref{prop:hpinjects}, and could possibly be weakened at
the expense of a more complicated proof for this proposition.  
Condition~5 is used only to establish
that the kernel $K_{S,T}$ is non-trivial, and could be replaced by the
more complicated condition: `there is a closed path in $K^1$ whose
intersection with the boundary of each polygon of $K$ is less than
half the length of that polygon' without changing the proof.  Note
that a single 13-gon satisfies all of the conditions except
Condition~5, so this condition cannot be omitted altogether.

The constants in the definition of a spectacular complex are chosen
to ensure the $C'(1/6)$ graphical small cancellation condition in
the associated group presentations.  We
have been unable to find analogous complexes that give rise to
presentations satisfying the $C'(1/n)$ condition for arbitrarily
large $n$.  The technique used in Section~\ref{sec:complex}
will produce suitable analogues of $K_1$, provided that $d$
is chosen to be at least $n+1$.  However, for large values of $d$
we have been unable to find a set of polygons to remove to
leave an acyclic subcomplex.

\

\leftline{\bf Authors' addresses:}

\obeylines

\smallskip
{\tt i.j.leary@soton.ac.uk} \qquad {\tt t.brown@soton.ac.uk} 

\smallskip
CGTA, 
School of Mathematical Sciences, 
University of Southampton, 
Southampton,
SO17 1BJ


\begin{thebibliography}{19} 

\bibitem{aschseg} M. Aschbacher and Y. Segev, \emph{A fixed point theorem
for groups acting on finite $2$-dimensional acyclic simplicial complexes}, 
Proc. London Math. Soc. \textbf{67} (1993) 329--354.  

\bibitem{bb} M. Bestvina and N. Brady, \emph{Morse theory and
  finiteness properties of groups}, Invent. Math. \textbf{129} (1997)
  445--470. 

\bibitem{bie} R. Bieri, \emph{Homological dimension of discrete
  groups}, second edition, Queen Mary College Mathematical Notes, 
Queen Mary College, Department of Pure Mathematics, London (1981).  

\bibitem{bhb} B. H. Bowditch, \emph{Continuously many quasi-isometry
  classes of $2$-generator groups}, Comment. Math. Helv. \textbf{73}
  (1998) 232--236.  

\bibitem{brady} N. Brady, \emph{Branched coverings of cubical
  complexes and subgroups of hyperbolic groups}, J. London
  Math. Soc. (2) \textbf{60} (1999) 461--480.  

\bibitem{brihae} M. R. Bridson and A. Haefliger, \emph{Metric spaces of
non-positive curvature}, Grundlehren der mathematischen Wissenschaften
  \textbf{319}, Springer-Verlag, New York-Berlin (1999).  
  
\bibitem{brown} K. S. Brown, \emph{Finiteness properties of groups}, 
J. Pure Appl. Algebra \textbf{44} (1987), 45--75.  

\bibitem{brobook} K. S. Brown, \emph{Cohomology of Groups}, 
Graduate Texts in Mathematics \textbf{87}, Springer-Verlag, 
New York-Berlin (1982).  

\bibitem{brownphd} T. Brown, \emph{Uncountably many quasi-isometry classes
  of groups of type $FP$ via graphical small cancellation theory}, PhD thesis,
  University of Southampton (2021).  

\bibitem{etc} K.-U. Bux, M. G. Fluch, M. Marschler, S. Witzel and
  M. C. B. Zaremsky, \emph{The braided Thompson's groups
    are $F_\infty$.  With an appendix by Zaremsky}, J. reine
  angew. Math. \textbf{718} (2016) 59--101.  


\bibitem{buxgon} K.-U. Bux and C. Gonzalez, \emph{The Bestvina--Brady
construction revisited: geometric computation of $\Sigma$-invariants 
for right-angled Artin groups}, J. London Math. Soc. (2) \textbf{60}
(1999) 793--801.  

\bibitem{buildings} K.-U. Bux, R. K\"ohl and S. Witzel, \emph{Higher 
finiteness properties of reductive arithmetic groups in positive 
characteristic: the Rank Theorem}, Ann. Math. \textbf{177} (2013) 
311--366.  


\bibitem{davpd} M. W. Davis, \emph{The cohomology of a Coxeter group
  with group ring coefficients}, Duke Math. J. \textbf{91} (1998) 
297--314.  

\bibitem{davbook} M. W. Davis, \emph{The geometry and topology of 
Coxeter groups}, London Mathematical Society Monographs Series
\textbf{32} Princeton Univ. Press, Princeton, NJ (2008).  

\bibitem{dicksleary} W. Dicks and I. J. Leary, \emph{Presentations
for subgroups of Artin groups}, Proc. Amer. Math. Soc. \textbf{127} 
(1999) 343--348.  

\bibitem{dror} E. Dror Farjoun, \emph{Fundamental group of homotopy
colimits}, Adv. Math. \textbf{182} (2004) 1--27. 
  
\bibitem{gromov} M. Gromov, \emph{Random walks in random groups}, 
  Geom. Funct. Anal. \textbf{13} (2003) 73--146.

\bibitem{gruber} D. Gruber, \emph{Groups with graphical $C(6)$ and
  $C(7)$ small cancellation presentations}, Trans. Amer. Math.
  Soc. \textbf{367} (2015), 2051--2078.  
  
\bibitem{hatcher} A. Hatcher, \emph{Algebraic Topology}, Cambridge University
  Press, Cambridge (2002).  

\bibitem{hig} G. Higman, \emph{Subgroups of finitely presented groups}, 
Proc. Roy. Soc. London Ser. A \textbf{262} (1961) 455--475.  

\bibitem{kimroush} K. H. Kim and F. W. Roush, \emph{Homology of certain
algebras defined by graphs}, J. Pure and Appl. Alg. \textbf{17} (1980)
  179--186.  

\bibitem{kls} R. P. Kropholler, I. J. Leary and I. Soroko,
  \emph{Uncountably many quasi-isometry classes of groups of type $FP$},
  Amer. J. Math. \textbf{142} (2020) 1931--1944.  

\bibitem{fpg} I. J. Leary, \emph{Uncountably many groups of type
  $FP$}, Proc. London Math. Soc. (3) \textbf{117} (2018) 246--276.  

\bibitem{sgfp2} I. J. Leary, \emph{Subgroups of almost finitely 
presented groups}, Math. Ann. \textbf{372} (2018) 1383--1391.  

\bibitem{vfg} I. J. Leary and B. E. A. Nucinkis, \emph{Some groups
of type $VF$}, Invent. Math. \textbf{151} (2003) 135--165.  

\bibitem{ijlrs} I. J. Leary and R. Stancu, \emph{Realising fusion
systems}, Algebra Number Theory \textbf{1} (2007) 17--34.  
  
\bibitem{lynsch} R. C. Lyndon and P. E. Schupp, \emph{Combinatorial
group theory}, Classics in Mathematics, Springer-Verlag, Berlin (2001).  

\bibitem{mow} A. Minasyan, D. Osin and S. Witzel, \emph{Quasi-isometric
  diversity of marked groups}, J. Topol. \textbf{14} (2021) 488--503.  
  
\bibitem{ollivier} Y. Ollivier, \emph{On a small cancellation 
theorem of Gromov}, Bull. Belg. Math. Soc. Simon Stevin \textbf{13} 
(2006) 75--89.  

\bibitem{sapir} M. Sapir, \emph{A Higman embedding preserving
asphericity}, J. Amer. Math. Soc. \textbf{27} (2014) 1--42.  

\bibitem{scottwall} P. Scott and C. T. C. Wall, \emph{Topological methods
  in group theory}, Homological Group Theory 137--203, London Math. Soc.
  Lecture Notes \textbf{36} Cambridge University Press, Cambridge (1979).  
  
\bibitem{sta} J. Stallings, \emph{A finitely presented group whose
$3$-dimensional integral homology is not finitely generated}, 
Amer. J. Math. \textbf{85} (1963) 541--543.  

\bibitem{swz} R. Skipper, S. Witzel and M. C. B. Zaremsky, \emph{Simple
groups separated by finiteness properties}, Invent. Math. \textbf{215} 
(2019) 713--740.  

\bibitem{thurston} W. P. Thurston, \emph{Three-dimensional geometry
  and topology. {V}ol. 1}, Princeton Mathematical Series \textbf{35},
  Princeton University Press, Princeton NJ (1997).  
  
\end{thebibliography}
\end{document}